\documentclass[11pt, letterpaper]{amsart}
\usepackage{hyperref}
\usepackage{mathtools}
\usepackage{amsthm,amssymb,amsfonts}
\usepackage[utf8]{inputenc}
\usepackage{float}
\usepackage{verbatim}
\usepackage{enumitem}
\usepackage{booktabs}
\usepackage{comment}
\usepackage{ mathrsfs }
\usepackage[mathscr]{eucal}
\usepackage{tikz}

\numberwithin{equation}{section}
\newtheorem{theorem}{Theorem}[section]

\newtheorem{prop}[theorem]{Proposition}

\newtheorem{corollary}[theorem]{Corollary}

\newtheorem{question}{Question}
\newtheorem{lemma}[theorem]{Lemma}

\theoremstyle{definition}
\newtheorem{definition}[theorem]{Definition}
\theoremstyle{remark}
\newtheorem{remark}[theorem]{Remark}

\newtheorem{notation}[theorem]{Notation}
\newtheorem{convention}[theorem]{Convention}
\newcommand{\N}{\mathbb{N}}
\newcommand{\R}{\mathbb{R}}
\newcommand{\C}{\mathbb{C}}

\newcommand{\Z}{\mathbb{Z}}
\newcommand{\T}{\mathbb{T}}

\newcommand{\B}{\mathbb{B}}

\newcommand{\Lcal}{\mathcal{L}}

\newcommand{\Sph}{\mathbb{S}}
\newcommand{\ghat}{\hat{g}}

\newcommand{\dbold}{{\mathbf{d}}}

\newcommand{\psicut}{\psi_{\mathrm{cut}}}
\newcommand{\Phip}{\Phi'}
\newcommand{\Psibold}{{\boldsymbol{\Psi}}}
\newcommand{\Ghat}{\hat{G}}
\newcommand{\sym}{\mathrm{sym}}
\newcommand{\rot}{\mathrm{rot}}
\newcommand{\gtilde}{\tilde{g}}
\newcommand{\Lcir}{L_{cir}}
\newcommand{\xx}{\mathrm{x}}

\newcommand{\zz}{\mathrm{z}}
\newcommand{\Phat}{\hat{\Phi}}
\newcommand{\ave}{\mathrm{avg}}
\newcommand{\osc}{\mathrm{osc}}
\newcommand{\sech}{\mathrm{sech}}
\newcommand{\shat}{\hat{\mathrm{s}}}
\newcommand{\group}{\mathscr{G}_m}

\newcommand{\loc}{\mathrm{loc}}
\newcommand{\Ltilde}{\tilde{\Lcal}}

\newcommand{\Mcal}{\mathcal{M}}
\newcommand{\vhat}{\hat{v}}
\newcommand{\Hcal}{\mathcal{H}}

\newcommand{\zhat}{\hat{z}}
\newcommand{\rhat}{\hat{r}}

\newcommand{\what}{\hat{w}}

\newcommand{\Bcal}{\mathcal{B}}
\newcommand{\uhat}{\hat{u}}

\newcommand{\phierr}{\varphi^{err}}
\newcommand{\Rcal}{\mathcal{R}}

\newcommand{\Jcal}{\mathcal{J}}
\newcommand{\Rcalappr}{\tilde{\mathcal{R}}}

\newcommand{\high}{\mathrm{high}}
\newcommand{\low}{\mathrm{low}}
\newcommand{\Happr}{\tilde{H}_{\Lcal}}
\newcommand{\Hlow}{\tilde{H}^{\low}_{\Lcal}}
\newcommand{\Hhigh}{\tilde{H}^{\high}_{\Lcal}}

\newcommand{\varphibd}{\varphi^\partial}

\newcommand{\nuhat}{\hat{\nu}}
\newcommand{\partialnu}{\nuhat}

\newcommand{\Lcalhat}{\hat{\Lcal}}

\newcommand{\Lcaltilde}{\tilde{\Lcal}}
\newcommand{\gcir}{\mathring{g}}
\newcommand{\Hdelta}{H_{\Delta}}
\newcommand{\Hdeltatilde}{\tilde{H}_{\Delta}}
\newcommand{\Hdeltatildep}{\tilde{H}^{err}_{\Delta}}
\newcommand{\xiv}{\xi_v}
\newcommand{\xivtilde}{\tilde{\xi}_v}
\newcommand{\xitilde}{\tilde{\xi}}
\newcommand{\cunder}{{c}}
\newcommand{\utilde}{\tilde{u}}
\newcommand{\Etilde}{\tilde{E}}
\newcommand{\Jtilde}{\tilde{J}}
\newcommand{\uC}{{\mathscr{C}}}
\newcommand{\Hscr}{\mathscr{H}}
\newcommand{\groupT}{\mathscr{G}_\T}
\newcommand{\SSS}{\mathsf{S}}
\newcommand{\XXXu}{\underline{\mathsf{X}}}
\newcommand{\YYYu}{\underline{\mathsf{Y}}}

\newcommand{\phiunder}{\underline{\phi}}
\newcommand{\junder}{\underline{j}}
\newcommand{\cutoff}[3]{\mathbf{\Psi}\left[ #1,#2;#3 \right]}
\newcommand{\Htide}{\tilde{H}_{\Delta}}
\newcommand{\Htider}{\tilde{H}^{err}_{\Delta}}

\begin{document}
\title[New Solutions for the Free Boundary Problem]{New Homogeneous Solutions for the One-Phase Free Boundary Problem}

\author[C.~Hines]{Coleman~Hines}
\author[J.~Kolesar]{James~Kolesar}
\author[P.~McGrath]{Peter~McGrath}
\address{Department of Mathematics, North Carolina State University, Raleigh NC 27695} 
\email{cthines@ncsu.edu}
\email{pjmcgrat@ncsu.edu}
\email{jnkolesa@ncsu.edu}
\begin{abstract}
For each sufficiently large integer $k$, we construct a domain in the round $2$-sphere with $k$ boundary components which is the link of a cone in $\R^3$ admitting a homogeneous solution to the one-phase free boundary problem.  This answers a question of Jerison-Kamburov, and also disproves a conjecture of Souam left open in earlier work. The method exploits a new connection with minimal surfaces, which we also use to construct an infinite family of homogeneous solutions in dimension four.

\end{abstract}
\maketitle
\section{Introduction}
\label{Sintro}

The one-phase free boundary problem has been studied with great success \cite{AltCaff,JerisonK2, JerisonK, JerisonS, JerisonD, Helein, Traizet, EngelsteinDuke, EngelsteinCrelle} using methods from minimal surface theory.  
This article adds to this tradition by identifying a new link to the minimal surfaces literature and using it as a starting point for constructing new homogeneous solutions for the one-phase problem.  These examples answer a question of Jerison-Kamburov \cite{JerisonK}, and also provide the first examples of domains in $\Sph^2$ which are extremals for the first Laplace eigenvalue but are not rotationally symmetric, disproving a conjecture of Hong \cite[p. 4014]{HongODE} and part of a conjecture of Souam \cite{Souam} left open in earlier work \cite{FallSchiffer}.

To fix notation, we recall the \emph{one-phase free boundary problem} is 
\begin{equation}
\label{E1phase}
\Delta u = 0 \quad \text{in} \quad U \cap \{u > 0 \}, 
\qquad
u = 0, \,\, \,  |\nabla u | = 1
\quad 
\text{on} 
\quad U \cap \partial \{ u > 0 \},
\end{equation}
where $U \subset \R^n$ is a domain.  Formally, \eqref{E1phase} is the Euler-Lagrange equation for the Alt-Caffarelli functional
\begin{align}
\label{Ealt}
J(u, U) = \int_U ( |\nabla u|^2 + \chi_{\{ u> 0\} } ) dx, 
\quad
u \in H^1(U,\R_+),
\end{align}
and Alt-Caffarelli's seminal work \cite{AltCaff} studied regularity for minimizers of the energy in \eqref{Ealt}---in general \eqref{E1phase} must be interpreted in a weak sense---using techniques inspired from regularity theory for minimal surfaces.  In analogy to the blow-up  method for reducing regularity questions for area-minimizing hypersurfaces in $\R^n$ to corresponding questions for area-minimizing cones in $\R^n$, the regularity for minimizers of \eqref{Ealt} can be studied through a blow-up procedure (see \cite{Weiss}) under which a sequence of rescalings of \eqref{E1phase} gives rise to a cone $\Omega \subset \R^n$ and a one-homogeneous function $u: \overline{\Omega} \rightarrow \R$ which is positive on $\Omega$ and solves 
\begin{equation}
\Delta u = 0
\quad
\text{in} 
\quad
\Omega, 
\qquad 
u = 0 \text{ and } |\nabla u| = 1 
\quad
\text{on } \partial \Omega \setminus \{ 0\}.
\end{equation}
Such a function is called a \emph{homogeneous solution} to the one-phase free boundary problem, and $\partial \Omega$ is called the \emph{free boundary}.

In dimensions $n=3,4,5,6,7$, an area-minimizing hypercone in $\R^n$ must be flat, hence smooth, while in dimension $n = 8$, Bombieri-De Giorgi-Guisti \cite{Bombieri} showed the singular Simons cone \cite{Simons} is area-minimizing. 
Similarly, when $n=3,4$, the free boundary associated to an energy-minimizing homogeneous solution to the one-phase problem must be flat \cite{JerisonCaff, JerisonS}, while when $n=7$, De Silva-Jerison \cite{JerisonD} found a singular minimizing free boundary. 

There is a well-known equivalence between homogenous solutions and a shape-optimization problem \cite{ElSoufi, PacardSic} for domains in the unit sphere $\Sph^{n-1}$: if $u$ is a homogeneous solution, its restriction $v$ to $\Omega_{\Sph} : = \Omega \cap \Sph^{n-1}$ solves 
\begin{equation}
\label{EoverOG}
\begin{cases}
v \geq 0 \hfill \quad &\text{on} \quad \Omega_\Sph\\
\Delta v + (n-1) v = 0 \hfill \quad &\text{on} \quad \Omega_\Sph\\
v = 0 \hfill \quad &\text{on} \quad \partial \Omega_\Sph \\
|\nabla v| = c \hfill \quad &\text{on} \quad \partial \Omega_\Sph,
\end{cases}
\end{equation}
and if $v$ satisfies  \eqref{EoverOG} on $\Omega_\Sph \subset \Sph^{n-1}$, then its one-homogeneous extension $u(x): = |x| v(x/|x|)$ defined on the cone $\Omega \subset \R^n$ over $\Omega_\Sph$ is a homogeneous solution. 
The first three items of \eqref{EoverOG} assert that $v$ is a first Dirichlet eigenfunction  for the Laplacian with eigenvalue $\lambda_1(\Omega_\Sph)=n-1$, while the 
condition $|\nabla v| = c$  is the Euler-Lagrange equation \cite{ElSoufi, PacardSic} for the functional $\Omega_\Sph \mapsto \lambda_1(\Omega_{\Sph})$ assigning a domain to its first eigenvalue, with a constraint on the enclosed area.  Domains $\Omega_\Sph$ admitting solutions to \eqref{EoverOG} are thus critical points for the first eigenvalue of the Laplacian and by convention are called \emph{extremal domains}.

Little is known about homogeneous solutions and solutions to \eqref{EoverOG} in general.  While in high dimensions, ODE or Lie group reduction methods give rise to a handful of examples  with continuous symmetry \cite{HongODE, Shklover, Karlovitz}, in dimension $3$, there are just two known examples up to rigid motions, each with rotational symmetry: the half-space solution, whose domain $\Omega_\Sph$ is a hemisphere, and an example due to Alt-Caffarelli \cite{AltCaff}, where $\Omega_\Sph$ is a tubular neighborhood of an equator circle.

While Souam conjectured \cite[Conjecture 1.1]{Souam} these are the \emph{only} examples in dimension $3$, Jerison-Kamburov asked \cite[Question 8.8]{JerisonK} the following:
\begin{question}[Jerison-Kamburov] 
Are there entire homogeneous solutions to the one-phase free boundary problem in $\R^3$ that are symmetric with respect to the discrete $\Z_n$ action around an axis, in analogy to the case of the Alt-Caffarelli example with $\Z_2$ symmetry?
\end{question}

We address this question, as well as Souam's conjecture, as follows.

\begin{theorem}
\label{Tmainint}
For each $m \in \N$ large enough, there is a domain $\Omega_m \subset \Sph^2$ with real-analytic boundary satisfying the following properties:
\begin{enumerate}
\item $\Omega_m$ admits a solution to the overdetermined problem \eqref{EoverOG}.
\item $\Omega_m$ is invariant under the action of $D_m \times \Z_2\leq O(3)$ on $\Sph^2$; \newline here $D_m$ is the dihedral group of order $2m$. 
\item $\Omega_m$ has $m+2$ boundary components, and is a perturbation of
	\begin{align*}
	\Sph^2 \setminus \left( D_{L_0}(\tau_0) \cup D_{L_2}(\tau_2) \right);
	\end{align*}
	here $D_{L_i}(\tau_i)$ denotes the $\tau_i$-neighborhood in $\Sph^2$ of the set $L_i$, where
	\begin{align*}
	L_0: = \{ ( \cos \textstyle{\frac{2\pi j}{m}}, \sin \textstyle{\frac{2\pi j}{m}} , 0 ) \}_{j\in \Z},
	\quad
	L_2: = \{ (0, 0, \pm 1)\}, 
	\\
	\tau_0 := m^{-3/4} e^{- \sqrt{m/2} + c_0},
	\quad
	\tau_2 := m^{-1/4} e^{-\sqrt{m/2} + c_2},
	\end{align*}
	and $c_0, c_2 \in \R$ are bounded by a constant independent of $m$. 
\end{enumerate}
\end{theorem}
In particular, the domain $\Omega_m$ in Theorem \ref{Tmainint} is the complement of a perturbed collection of $m+2$ small geodesic disks, $m$ of which have centers arranged symmetrically on an equatorial circle, and two of which have centers at the north and south poles of the sphere.  We emphasize that $\Omega_m$ has a finite isometry group, and hence cannot arise from ODE reduction methods as in \cite{HongODE, Shklover, Karlovitz}.

Although a simpler $D_m\times \Z_2$-symmetric candidate domain would be one whose complement is a perturbed collection of $m$ geodesic disks arranged around the equator (and no disks removed at the poles), it turns out there is no solution of \eqref{EoverOG} arising in this way.  We \emph{do} expect a $D_m \times \Z_2$-symmetric solution domain isotopic to the preceding, but with excised topological disks which are \emph{not} approximately round, instead modeled on the ones associated to the \emph{Scherk solutions} discussed in \cite[Section 7]{JerisonK}.

As will be described later, ideas in the proof of Theorem \ref{Tmainint} can be applied in any dimension, and we carry out the following construction when $n=4$.

\begin{theorem}
\label{Tmain4}
For each $m \in \N$ large enough, there is a domain $\Omega_m \subset \Sph^3$ with real-analytic boundary satisfying the following properties:
\begin{enumerate}
\item $\Omega_m$ admits a solution to the overdetermined problem \eqref{EoverOG}. 
\item $\Omega_m$ has $m^2$ boundary components, and is a perturbation of 	$\Sph^3 \setminus D_L(\tau)$;
	here $D_L(\tau)$ denotes the $\tau$-neighborhood in $\Sph^3$ of the set $L$, where
	\begin{align*}
	\begin{gathered}
	L = \Big\{ \frac{1}{\sqrt{2}} ( e ^{i \frac{2\pi j}{m}} , e ^{i \frac{2\pi k}{m}} )\Big \}_{j, k \in \Z} \subset \Sph^3,
	\quad
	\tau = \left( \frac{m^2}{\pi F} + cm\right)^{-1}, 
	\end{gathered}
	\end{align*}
	$F \in (2.18, 2.19)$, and $c \in \R$ is bounded independently of $m$. 
\item $\Omega_m$ is invariant under the stabilizer of $L$ in $O(4)$. 
\end{enumerate}
\end{theorem}

In both Theorem \ref{Tmainint} and Theorem \ref{Tmain4}, we note that the components of $\partial \Omega_m$ do not have constant mean curvature, disproving a conjecture of Hong \cite[p. 4014]{HongODE}.  In light of classification results \cite{JerisonS}, the examples constructed in this article are not stable for the Alt-Caffarelli functional, and may be of interest in relation to work \cite{Kriventsov, KamburovMorse} on the regularity and Morse index of non-minimizing solutions to the one-phase problem.

\subsection{Overdetermined elliptic problems}
The literature on elliptic problems such as \eqref{EoverOG} with overdetermined boundary conditions is vast.  For a general class of such problems, Serrin \cite{Serrin} proved that a bounded domain in $\R^n$ admitting a solution must be a ball; in particular, his work implies
a domain $\Omega \subset \R^n$ admitting a solution to \eqref{EoverOG} must be a ball, so that the only $\lambda_1$-extremal domains in $\R^n$ are round.

There has been significant interest in extending Serrin's work, either for unbounded domains, or for overdetermined problems in Riemannian manifolds such as the spheres $\Sph^n$ or hyperbolic spaces $\mathbb{H}^n$.  For example, Berestycki-Caffarelli-Nirenberg conjectured \cite{BCN} that an unbounded domain in $\R^n$ admitting a bounded solution to one of a class of  overdetermined semilinear elliptic equations must be a half-space, the complement of a ball, or a generalized-cylinder, and Souam \cite{Souam} conjectured that a domain $\Omega_\Sph \subset \Sph^2$ solving \eqref{EoverOG} must be a hemisphere or a round, symmetric annulus.  Various authors \cite{Kumaresan, Ciraolo, Molzon, Brock} have also proved Serrin-type theorems for domains in $\mathbb{H}^n$ or in $\Sph^n$, although results in the $\Sph^n$ case typically impose additional assumptions, such as the condition that the domains under consideration lie in a hemisphere.

In recent years, it has become clear that Serrin-type rigidity need not hold in such regimes, and bifurcation methods have become a popular tool for constructing non-round domains admitting solutions to overdetermined elliptic problems.  For example, Sicbaldi \cite{SicbaldiTori} and Sicbaldi-Ros-Ruiz \cite{SicbaldiRos} found counterexamples to the Berestycki-Caffarelli-Nirenberg conjecture, Fall-Minlend-Weth \cite{FallSchiffer} disproved a general version of Souam's conjecture mentioned earlier,
and other authors have found non-round domains admitting solutions to overdetermined problems either for unbounded domains, or for domains in Riemannian manifolds \cite{SicbaldiDuke, Dai, FallSphere, WethSerrin, SicbaldiTori, KamburovAnn, Morabito}.  
P. Sicbaldi has also informed us that for certain values of $\lambda_1$, there are $\lambda_1$-extremal domains in $\Sph^2$ bifurcating from tubular neighborhoods of the equator circle. 

An idiosyncratic aspect of the problem \eqref{EoverOG} studied in this article is that bifurcation methods are not expected to provide nontrivial solutions.  To see this, we recall \cite{Reznikov, Souam, Nadirashvili} a correspondence between \eqref{EoverOG} and the \emph{free boundary minimal surfaces} in the Euclidean $3$-ball $\B^3$, which are the minimal surfaces in $\B^3$ which meet the sphere $\Sph^3 = \partial \B^3$ orthogonally: when $n=2$, each solution to \eqref{EoverOG} gives rise to a branched minimal immersion in $\B^3$ with free boundary, and conversely, the Gauss map on a free boundary minimal surface in $\B^3$ gives rise to a branched immersed domain $\Omega_\Sph \subset \Sph^2$ on which the pushforward of the support function $\langle X, \nu \rangle$ solves \eqref{EoverOG}.  Here $X$ and $\nu$ are the position and unit normal vector fields on the surface, respectively.

Bifurcation methods in line with the ones used in \cite{WethSerrin, FallSphere, KamburovAnn, FallSchiffer, SicbaldiTori} would be expected to produce annular domains bifurcating from a rotationally-symmetric annulus, and would have $\Z_2 \times D_m$ symmetry for some $m$.  On the other hand, under the preceding correspondence, the problem \eqref{EoverOG} for annular domains is equivalent \cite{LeeYeon, Espinar} to the question of whether the critical catenoid is only embedded free boundary minimal surface in the $3$-ball, and this rigidity is known to hold with additional discrete symmetry \cite{KusnerMcGrath, Seo} consistent with $\Z_2\times D_m$-symmetry.

\subsection{Outline of the method} The basis for our construction is an observation that a class of minimal surfaces in the round $3$-sphere $\Sph^3$, the \emph{minimal doublings} of the equator $\Sph^2 = \{ x \in \Sph^3 : x_4 =0\}$, have canonical domains with solutions to \eqref{EoverOG}. As in \cite[Definition 1.1]{LDG}, a surface $M \subset \Sph^3$ is a \emph{doubling} of $\Sph^2$ if the nearest-point projection $\pi$ to $\Sph^2$ is well-defined on $M$ and $M = M_1 \cup M_2$, where $M_1$ is a $1$-manifold, $M_2 \subset M$ is open, $\pi|_{M_1}$ is a diffeomorphism, and $\pi|_{M_2}$ is a $2$-sheeted covering map.  The doubling $M$ is called \emph{minimal} if $M$ is a minimal surface, and \emph{side-symmetric} if $M$ is invariant under the reflection of $\Sph^3$ fixing $\Sph^2$ pointwise.  

The observation is simply that the coordinate function $x_4$ solves \eqref{EoverOG} on the ``top-half" $M_+ = \{x\in M : x_4 > 0\}$ of a side-symmetric minimal $\Sph^2$ doubling, where the Laplacian and gradient are computed with respect to the metric on $M$.
 Though $M_+$ is \emph{not} a domain in $\Sph^2$, there are \cite{KapSph, CPAM, LDG, KKMS} families of side-symmetric minimal doublings which converge as varifolds to $\Sph^2$ with multiplicity two, 
 providing examples where $M_+$ is \emph{arbitrarily close} to a domain of $\Sph^2$ in a Hausdorff sense.  (A large-genus surface from such a family resembles two nearby and approximately parallel copies of $\Sph^2$ joined by many small, approximately catenoidal bridges.) This suggests the possibility of solutions for \eqref{EoverOG} whose domains $\Omega_\Sph \subset \Sph^2$ are perturbations of projections $\pi(M) \subset \Sph^2$ of minimal $\Sph^2$-doublings. 

This heuristic may be made more precise by introducing terminology from Kapouleas's \emph{Linearized Doubling} (LD) approach \cite{KapSph, CPAM, LDG} for constructing minimal doublings.  In this approach, the region of an $\Sph^2$-doubling away from the catenoidal bridges is a perturbation of the graphs of $\pm \varphi$, where $\varphi$ is an \emph{LD solution}, solving the Jacobi equation $(\Delta + 2) \varphi = 0$ with prescribed logarithmic singularities at a finite \emph{singular set} $L \subset \Sph^2$.  The singularity strength $\tau_p$ at a given $p \in L$ corresponds to the waist radius of the catenoidal bridge centered at $p$, and part of the method constructs smooth \emph{initial surfaces} obtained from gluing such bridges to the graphs of suitable LD solutions $\varphi$.  In order to perturb an initial surface to exact minimality, the LD solution $\varphi$ must approximately satisfy certain \emph{matching equations} $\Mcal_p \varphi = 0, p \in L$ related to the asympotics of the catenoidal bridges.  

From the point of view of this article, an LD solution $\varphi$ already satisfies the first three conditions in \eqref{EoverOG} on its domain $\{ \varphi > 0\}$ of positivity.  If additionally $\Mcal_p \varphi = 0$ at $p \in L$ in the sense of \cite[Definition 3.1]{LDG}, then $\varphi$ satisfies an expansion of the form
\begin{align}
\label{EintroM}
\varphi (r) = \tau_p \log (2 r/ \tau_p) + O(r^2 |\log r| ),
\end{align}
near $p$, where $r$ is the distance function from $p$.  In particular \eqref{EintroM} shows that an LD solution satisfying these matching equations approximately satisfies the last condition in \eqref{EoverOG}.

While each minimal $\Sph^2$-doubling $M$ with large enough genus is expected to correspond to a solution of \eqref{EoverOG} whose domain is a perturbation of $\pi(M)$, in order to minimize technical difficulties, in this article we focus only on the ``equator-poles" family \cite[Section 6]{KapSph}, whose corresponding LD solutions have singular sets $L$ which are fixed by the symmetries (unlike other families of LD solutions \cite{CPAM, LDG}) and consist of $m+2$ points, $m$ of which are symmetrically arranged around the equator, and two of which are at the north and south poles. The family of LD solutions still has two free parameters, corresponding to the  strengths of the singularities at the poles and equator, and the ratio of these strengths must be chosen judiciously in order to satisfy the matching equations.

Given an LD solution $\varphi$ in this family, part of the approach for perturbing $\{ \varphi > 0\}$ to an exact solution of \eqref{EoverOG} follows elegant methodology developed by Pacard-Sicbaldi and Sicbaldi in \cite{PacardSic, SicbaldiOut}, where it was shown that the interior or exterior of a small geodesic ball in a Riemannian manifold can in certain cases be perturbed to an $\lambda_1$-extremal domain.  This approach has since been used in many other constructions for domains admitting solutions to overdetermined problems \cite{FallSphere, WethSerrin, FallSchiffer, SicbaldiTori, SicbaldiGeneral,  Morabito, KamburovAnn}.  Some modifications of this methodology are necessary, because the constructions in this article depend on a discrete parameter, and it is not possible to use the implicit function theorem directly.  Also, like in \cite{KamburovAnn}, but unlike some of the other preceding problems, the construction in the proof of Theorem \ref{Tmainint} still has one free parameter after accounting for the scaling invariance of \eqref{EoverOG}, which adds additional technical subtlety. 

Finally, just as the LD methodology can be applied  in ambient dimensions bigger than three \cite{KapZou}, so too can the methods used in the proof of Theorem \ref{Tmain}.  In particular, starting with the LD solutions Kapouleas-Zou use \cite{KapZou} for constructing minimal $\Sph^3$-doublings in $\Sph^4$, we adapt the methodology just described to prove Theorem \ref{Tmain}.

\subsection{Outline of the paper}
Section \ref{Snot} studies the geometry of the sphere $\Sph^2$, the symmetry groups $\group$, and properties of the operator $\Lcal = \Delta +2$ on $\Sph^2$.  Section \ref{SLD} concerns LD solutions:  after recalling definitions from \cite{KapSph, LDG}, most of the section is devoted to proving estimates, culminating in Lemma \ref{Lphitau}, on LD solutions $\varphi = \varphi[m]$ which provide approximate solutions to \eqref{EoverOG} on their sets $\{\varphi > 0 \}$ of positivity.  Although this family was studied in \cite[Section 6]{KapSph}, we need different and at times more detailed estimates, so we keep the discussion self-contained.

Section \ref{Slin} studies linear operators related to the domains which will later be perturbed to exact solutions for \eqref{EoverOG}.  These domains are of the form $\Sph^2 \setminus D$, where $D$ is a union of small geodesic disks centered at the points of $L$.  Proposition \ref{Pext1} introduces an $\Lcal$-extension operator $H_\Lcal$, determining a solution to the equation $\Lcal u = 0$ on $\Sph^2 \setminus D$ with appropriate prescribed data on the boundary $\partial D$.  While it is not in general possible to prescribe \emph{constant} data, we need to be able to prescribe \emph{locally constant} data on the different, up to symmetry, components of the boundary $\partial D$ which has average zero, and this is achieved by judiciously using facts about the LD solutions from Section \ref{SLD}.  The main purpose of the section is to construct a right-inverse for an operator $\Bcal$ defined on functions on $\partial D$ which is essentially a shifted Dirichlet-to-Neumann operator for $\Lcal$ on $\Sph^2 \setminus D$.  Analogous operators arise in other perturbation constructions for solutions of overdetermined problems \cite{PacardSic, SicbaldiTori, SicbaldiCylinder, SicbaldiOut, SicbaldiGeneral, KamburovAnn}.

Section \ref{Spert} studies perturbations of $\Sph^2 \setminus D$ whose boundaries project graphically onto $\partial D$:  Lemma \ref{Lphiv} and Proposition \ref{Pphiv} study perturbations preserving the condition $\lambda_1 = 2$ and give an expansion of the corresponding eigenfunction,  Proposition \ref{LinN} studies the normal derivative of these eigenfunctions and identifies the operator $\Bcal$ as carrying the dominant linear terms, and Theorem \ref{Tmain} contains the proof of the main theorem. Finally, Section \ref{Sdim4} contains the $4$-dimensional results, including the proof of Theorem \ref{Tmain4}.
\subsection{Acknowledgments}
While this work was carried out, P.M. was partially supported by Simons Foundation Collaboration Grant 838990.  P.M. is grateful to N. Kapouleas for many insightful discussions over the years, and to F. Abedin and M. Engelstein for helpful suggestions.

\section{Elementary Geometry and Notation}
\label{Snot}
\subsection{H{\"o}lder norms and cutoff functions}
\begin{definition}
\label{Dnorm}
Assuming that $\Omega$ is a domain inside a manifold, $g$ is a Riemannian metric on $\Omega$, $k \in \N_0, \alpha \in [0, 1)$, that $u \in C^{k, \alpha}_{\loc}(\Omega)$, $\rho f: \Omega \rightarrow (0, \infty)$ are given functions, and that the injectivity radius in the manifold around each point $x$ in the metric $\rho^{-2}(x) g$ is at least $1/10$,  we define
\begin{align*}
\| u : C^{k, \alpha}(\Omega, \rho, g,  f) \| : = \sup_{x \in \Omega} \frac{\| u : C^{k, \alpha}(\Omega \cap B_x , \rho^{-2}(x) g) \|}{f(x)},
\end{align*}
where $B_x$ is a geodesic ball that is centered at $x$ and of radius $1/100$ in the metric $\rho^{-2}(x)g$.  For simplicity we may omit any of $\alpha, \rho$, or $f$ when $\alpha = 0, \rho \equiv 1$, or $f \equiv 1$, respectively. 
\end{definition}
The following notation regarding cutoff functions is standard \cite{KapSph}.

\begin{definition}
\label{DPsi} 
Fix a smooth function $\Psi:\R\to[0,1]$ such that
\begin{enumerate}[label={(\roman*)}]
\item $\Psi$ is nondecreasing;
\item $\Psi\equiv1$ on $[1,\infty)$ and $\Psi\equiv0$ on $(-\infty,-1]$; and
\item $\Psi-\frac12$ is an odd function.
\end{enumerate}
\end{definition}
\noindent Given $a,b\in \R$ with $a\ne b$,
 define a smooth function
$\psicut[a,b]:\R\to[0,1]$
by
\begin{equation*}
\psicut[a,b]:=\Psi\circ L_{a,b},
\end{equation*}
where $L_{a,b}:\R\to\R$ is the linear function defined by the requirements $L(a)=-3$ and $L(b)=3$.
Note that $\psicut[a,b]$ has the following properties:

\begin{enumerate}[label={(\roman*)}]
\item $\psicut[a,b]$ is weakly monotone.

\item 
$\psicut[a,b]=1$ on a neighborhood of $b$; and 
\newline
$\psicut[a,b]=0$ on a neighborhood of $a$.

\item $\psicut[a,b]+\psicut[b,a]=1$ on $\R$.
\end{enumerate}

Suppose now we have real-valued functions $f_0,f_1$, and $\rho$ defined on some domain $\Omega$.
We define a new function 
\begin{equation}
\label{EPsibold}
\Psibold\left [a,b; \rho  \right](f_0,f_1):=
\psicut[a,b ]\circ \rho \, f_1
+
\psicut[b,a]\circ  \rho \, f_0.
\end{equation}
Note that
$\Psibold[a,b;\rho ](f_0,f_1)$
depends linearly on the pair $(f_0,f_1)$
and transits from $f_0$
on $\Omega_a$ to $f_1$ on $\Omega_b$,
where $\Omega_a$ and $\Omega_b$ are subsets of $\Omega$ which contain
$\rho^{-1}(a)$ and $\rho^{-1}(b)$ respectively,
and are defined by
$$
\Omega_a=\rho^{-1}\Big(\big(-\infty,a+\frac13(b-a)\big)\Big),
\qquad
\Omega_b=\rho^{-1}\Big(\big(b-\frac13(b-a),\infty\big)\Big),
$$
when $a<b$, and 
$$
\Omega_a=\rho^{-1}\Big(\big(a-\frac13(a-b),\infty\big)\Big),
\qquad
\Omega_b=\rho^{-1}\Big(\big(-\infty,b+\frac13(a-b)\big)\Big),
$$
when $b<a$.
Clearly if $f_0,f_1,$ and $\rho$ are smooth, then so is
$\Psibold[a,b;\rho ](f_0,f_1)$.

\subsection{The configurations and the symmetries}
Denote by $\Sph^2$ the unit sphere in $\R^3$, and by $g$ the round metric on $\Sph^2$ induced by the Euclidean metric on $\R^3$.  
\begin{notation}
For any $X \subset \Sph^2$, we write $\dbold_X$ for the Riemannian distance from $X$, and define the $\delta$-neighborhood of $X$ by
\[
D_X(\delta) : = \{ p \in \Sph^2 : \dbold_X(p) < \delta\}.
\]
If $X$ is finite we just list its points;   for example, $\dbold_q(p)$ is the geodesic distance between $p$ and $q$ and $D_q(\delta)$ is the geodesic disc with center $q$ and radius $\delta$. 
\end{notation}

Our constructions depend on a large number $m\in \N$, which we now fix.  Throughout, we will assume $m$ is as large as needed in terms of absolute constants. 

\begin{definition}
\label{dL}
\label{dgroup}
Define $L = L[m]: =  L_0 \cup L_2 \subset \Sph^2$, where
\begin{align*}
\begin{gathered}
L_0 = L_0[m]: = \big\{ ( \cos \textstyle{ \frac{2\pi k}{m}} , \sin \frac{2\pi k}{m}, 0): k \in \Z\big \}, 
\quad L_2 = \{ \pm (0, 0, 1) \}
\end{gathered}
\end{align*}
and let $\group$ denote the group of isometries of $\Sph^2$ fixing $L$.  We also define 
\begin{align*}
\begin{gathered}
p_0: = (1, 0, 0)\in L_0, 
\quad
p_2: = (0, 0, 1) \in L_2, 
\\ 
\Lcir : = \{ (x_1,x_2, x_3) \in \Sph^2 : x_3 = 0\}.
\end{gathered}
\end{align*}
Finally, whenever $\tau : L \rightarrow \R$ is a $\group$-symmetric function, we denote by $\tau_i$ the value $\tau$ attains on $L_i$.  
\end{definition}

\begin{notation}
\label{Nrot}
If $X$ is a  function space consisting of functions defined on a $\group$-invariant domain $\Omega \subset \Sph^2$, we use a subscript ``sym" to denote the subspace $X_{\sym} \subset X$ of $\group$-invariant functions. 
\end{notation}

\begin{lemma}[Properties of $\group$]
\label{Lgroup}
The following hold. 
\begin{enumerate}[label=\emph{(\roman*)}]
\item $\group$ is isomorphic to $D_{2m}\times \Z_2$, where $D_{2m}$ is the dihedral group of order $2m$. 
\item $\group$ is generated by reflections through great circles passing through points of $L$.
\item The derivative of any $\group$-symmetric differentiable function vanishes at each point of $L$. 
\end{enumerate}
\end{lemma}
\begin{proof}
This follows easily from Definition \ref{dL}, and we omit the details.
\end{proof}

\begin{definition}
\label{dharm}
If $S \subset \Sph^2$ is a round circle and $k \in \N$, denote by $\Hcal^k(S)$ the $k$-th nontrivial eigenspace for the Laplacian $\Delta_S$ on $S$, and by $\Hcal^0(S)$ the span of the constant functions on $S$. 
 If $X \subset \Sph^2$ is a finite, $\group$-invariant set of pairwise disjoint round circles, let
\begin{align*}
\Hcal^k_{\sym}(X) = \{ C^\infty_{\sym}(X)  : u|_S \in \Hcal^k(S) \quad \text{for each} \quad S \in X\} .
\end{align*}
\end{definition}
\begin{lemma}
\label{Lharm0}
Let $I_0, I_2 \subset \N$ be the subsets of indices defined by
\begin{align*}
I_0 = \{2, 4, \cdots\},
\quad
I_2 = \{2m, 4m, \cdots \}.
\end{align*}
If $i \in \{0, 2\}$ and $r \in (0, 1/m)$, then
\begin{enumerate}[label=\emph{(\roman*)}]
\item	 $\Hcal^k_{\sym}(\partial D_{L_i}(r))$ is $1$-dimensional for each $k \in I_i$, and \newline
	 $\Hcal^k_{\sym}(\partial D_{L_i}(r))$ is $0$-dimensional for each $k \in \N\setminus I_i$.
\item $L^2_{\sym}(\partial D_{L_i} (r)) = \Hcal^0_{\sym}(\partial D_{L_i}(r)) \oplus [\bigoplus_{k \in I_i} \Hcal^k_{\sym}( \partial D_{L_i}(r))] $.
\end{enumerate}
\end{lemma}
\begin{proof}
Given $r \in (0, 1/m)$, note that $\partial D_{L_0}(r) = \cup_{p \in L_0} \partial D_p(r)$ and that $\{ \partial D_p(r)\}_{p \in L_0}$ is a collection of $m$ circles on which $\group$ acts transitively.  It follows that $\Hcal^k_{\sym}(\partial D_{L_0}(r))$ is isomorphic to $\Hcal^k_{\sym}(\partial D_p(r))$ for any $p \in L_0$.  Since each circle of $\partial D_{L_0} (r)$ is invariant under two orthogonal reflections, then item (i) follows in the case $i=0$.  When $i=2$, the argument is very similar, and the details are omitted. 
Finally, (ii) follows easily from (i) and the definitions.  
\end{proof}
\subsection{The operator $\Lcal = \Delta + 2$}
\begin{definition}
Denote by $\Lcal$ the operator $\Delta + 2$ on $\Sph^2$, where $\Delta$ is the Laplace-Beltrami operator with respect to the usual metric $g$ on $\Sph^2$. 
\end{definition}

Throughout, we will use the fact that $\ker \Lcal$ is spanned by the coordinate functions on $\Sph^2$, and in particular that $(\ker \Lcal)_{\sym}$ is trivial.
\begin{lemma}
\label{LG}
The function $G \in C^\infty( ( 0, \pi) )$ defined by
\begin{align*}
G(r) = \cos r \log \left( 2 \tan\textstyle{ \frac{r}{2}}\right) + 1 - \cos r
\end{align*}
has the following properties:
\begin{enumerate}[label=\emph{(\roman*)}]
\item $\Lcal( G \circ \dbold_p) = 0$ on $\Sph^2 \setminus \{p, -p\}$, whenever $p \in \Sph^2$.
\item $G(r) = (1+ O(r^2)) \log r$ for small $r> 0$. 
\item $\| G - \cos r \log r : C^k ( ( 0, 1), r, dr^2, r^2 ) \| \leq C(k)$. 
\end{enumerate}
\end{lemma}
\begin{proof}
See \cite[Lemma 2.20]{KapSph}.
\end{proof}
\begin{remark}
\label{rk0}
Whenever $p \in \Sph^2$, the coordinate function $\cos \dbold_p \in C^\infty(\Sph^2)$ satisfies $\Lcal \cos \dbold_{p} = 0$. From this and Lemma \ref{LG},  it follows that the function $1+\cos \dbold_p \log \tan \frac{\dbold_p}{2} \in C^\infty(\Sph^2 \setminus \{p, -p\})$ solves $\Lcal u = 0$ on $\Sph^2 \setminus \{p, -p\}$. 
\end{remark}

It will be easier to state some of our estimates if we use a scaled metric $\gtilde$ on $\Sph^2$ and a scaled linear operator $\Lcaltilde$, defined by
\begin{align}
\label{Egtilde}
\gtilde : = m^2 g, 
\quad
\Ltilde : = \Delta_{\gtilde} + 2m^{-2} = m^{-2} \Lcal. 
\end{align}
\subsection{Rotationally invariant functions}
We call a $\group$-invariant function defined on a domain of $\Sph^2$ that depends only on the distance $\dbold_{\Lcir}$ to the equator circle $\Lcir$ a \emph{rotationally invariant function}.  Motivated by this, we introduce the following notation.
\begin{notation}
\label{Nrot2}
  If $\Omega$ is a $\group$-invariant union of parallel circles, and $X_{\sym}$ is a space of $\group$-invariant functions,  we use a subscript ``$\rot$" to denote the subspace $X_{\rot} \subset X$ of functions  which depend only on $\dbold_{\Lcir}$.  
\end{notation}
\begin{definition}
\label{dave}
Given a $\group$-invariant function $\varphi$ on a domain $\Omega \subset \Sph^2$, we define a rotationally invariant function $\varphi_{\ave}$ on the union $\Omega'$ of the parallel circles on which $\varphi$ is integrable (whether contained in $\Omega$ or not), by requesting that on each such circle $C$, 
\begin{align*}
\varphi_{\ave}|_{C} : = \mathrm{avg}_C \varphi. 
\end{align*}
We also define $\varphi_{\osc}$ on $\Omega \cap \Omega'$ by $\varphi_{\osc} : = \varphi - \varphi_{\ave}$. 
\end{definition}
\section{LD Solutions}
\label{SLD}
\subsection{Definitions}
We first recall the notion of an LD solution, from \cite{KapSph}.
\begin{definition}
\label{dLD}
We call $\varphi$ a \emph{linearized doubling (LD) solution} on $\Sph^2$ when there exists a finite set $L \subset \Sph^2$, called the \emph{singular set} of $\varphi$, and a function $\tau : L \rightarrow \R \setminus \{0\}$, called the \emph{configuration} of $\varphi$, satisfying the following, where $\tau_p$ denotes the value of $\tau$ at $p \in L$.  
\begin{enumerate}[label={(\roman*)}]
\item $\varphi \in C^\infty ( \Sph^2 \setminus L)$ and $\Lcal \varphi = 0$ on $\Sph^2 \setminus L$.  
\item For each $p \in L$, the function $\varphi - \tau_p \log \dbold_p$ is bounded on $\Sph^2 \setminus L$. 
\end{enumerate}
\end{definition}
\begin{lemma}[$\group$-symmetric LD solutions]
\label{Lldun}
For each $\group$-invariant finite set $L \subset \Sph^2$, and each $\group$-invariant function $\tau : L \rightarrow \R \setminus \{ 0\}$, there is a unique $\group$-invariant LD solution with configuration $\tau$.  
\end{lemma}
\begin{proof}
This is Lemma 3.10 of \cite{KapSph}. 
\end{proof}
In this article, we are interested in $\group$-symmetric LD solutions whose singular set is $L = L[m]$, where $L[m] = L_0 \cup L_2$ was defined in Definition \ref{dL}.  Furthermore, we are interested in LD solutions $\varphi$ which have vanishing \emph{mismatch}, $\Mcal_L \varphi = 0$, in the sense of the following lemma.  
\begin{lemma}[Mismatch]
\label{Lmismatch}
Given $\varphi, L$, and $\tau$ as in \ref{Lldun} with $L = L[m]$ as in \ref{dL} and $\tau_i> 0$ for $i = 0, 2$ (recall \ref{dL}), there are numbers $\Mcal_i \varphi \in \R$ called the \emph{mismatch of $\varphi$ at $L_i$}, defined by requesting that near $L_i$, 
\begin{align}
\label{Emismatch}
\varphi = \tau_i \log(\dbold_{L_i}/ \tau_i ) + \Mcal_i \varphi + O( \dbold^2_{L_i} |\log \dbold_{L_i} |).
\end{align}
The vector $\Mcal_L \varphi : = (\Mcal_0 \varphi, \Mcal_2 \varphi)$ is called the \emph{mismatch} of $\varphi$. 
\end{lemma}
\begin{proof}
For $i = 0, 2$, let $p_i$ be as in Definition \ref{dL}.  By Lemma 3.9 and Definition 3.10 of \cite{LDG}, there is an affine linear function $A_i$ such that
\begin{align*}
\varphi \circ \exp_{p_i}(v) = \tau_i \log ( |v|/ \tau_i) + A_i + O ( |v|^2 \log |v|)
\end{align*}
for all small $v \in T_{p_i}\Sph^2$.
On the other hand, the $\group$-symmetry implies the differential of $A_i$ must vanish; thus $A_i$ is a constant function.  In combination with the $\group$-symmetry, this completes the proof.
\end{proof}
Our study of the mismatch of $\group$-invariant LD solutions with singular set $L[m]$ depends on estimates of LD solutions $\Phi_0$ and $\Phi_2$, originally studied in \cite[Section 6]{KapSph}, which we now recall.
\begin{definition}
\label{dPhio}
Using Lemma \ref{Lldun}, define $\group$-symmetric LD solutions 
\begin{align*} 
\Phi_0 = \Phi_0[m] 
\quad
\text{and}
\quad
\Phi_2
\end{align*}
with configurations $\tau: L_i \rightarrow \R$ satisfying $\tau(p_i) = 1$ for $i=0, 2$.
\end{definition}
By Lemma \ref{Lldun}, note that a $\group$-symmetric LD solution $\varphi$ with singular set $L[m]$ can be decomposed as
\begin{align*}
\varphi = \tau_0 \Phi_0 + \tau_2 \Phi_2.
\end{align*}
\subsection{Estimates on $\Phi_0$}
\begin{lemma}
\label{LPhiave}
$\Phi_{0, \ave}: = (\Phi_0)_{\ave}  = \frac{m}{2} \sin \dbold_{\Lcir}$. 
\end{lemma}
\begin{proof}
Because $\Phi_{0,\ave}$ is smooth on each hemisphere (component of $\Sph^2 \setminus \Lcir$) and satisfies the equation $\Lcal \Phi_{0,\ave} = 0$ on $\Sph^2 \setminus L_0$, the rotational invariance implies $\Phi_{0,\ave} = A \sin \dbold_{\Lcir}$ for some $A \in \R$.  For $0< \epsilon_1 << \epsilon_2$, integrating $\Lcal \Phi_0 = 0$ on the domain $D_{\Lcir}(\epsilon_2) \setminus D_{L_0}(\epsilon_1)$, integrating by parts, and taking the limit as $\epsilon_1 \rightarrow 0$ first and then as $\epsilon_2 \rightarrow 0$, the logarithmic behavior near $L_0$ shows that $2\pi m = 4\pi A$, and the conclusion follows.
\end{proof}
We now define a decomposition $\Phi_0 = \Ghat_0 + \Phat_0 + \Phip_0$, where $\Ghat_0$ is supported near $L_0$ and contains $\Phi_0$'s singular part, $\Phat_0$ is smooth and rotationally invariant, and $\Phip_0$ is treated as an error term to be estimated. 
\begin{convention}
In what follows, we let $\delta = 1/(10m)$. 
\end{convention}
\begin{definition}
\label{dGhat}
Given $\Phi_0 = \Phi_0[m]$ as in \ref{dPhio}, we define
\begin{align*}
\Ghat_0  \in C^\infty_{\sym}(\Sph^2 \setminus L_0),
\quad
\Phat_0 \in C^\infty_{\rot }( \Sph^2), 
\quad
\Phip_0 \in C^\infty_{\sym}(\Sph^2),
\quad
E_0' \in C^\infty_{\sym}(\Sph^2)
\end{align*}
by requesting that $\Ghat_0$ is supported on $D_{L_0}(3 \delta)\setminus L_0$, where it is defined by
\begin{align}
\label{EGhat0}
\Ghat_0 = \Psibold[ 2\delta, 3 \delta; \dbold_{L_0}] ( G\circ \dbold_{L_0} - \log (1/m) \cos \dbold_{L_0}, 0) 
\quad
\text{on } D_{L_0}(3\delta),
\end{align}
that $\Phat_0 = \Phi_{0, \ave}$ on $\Sph^2 \setminus D_{\Lcir}(1/m)$, that
\begin{align*}
\Phat_0 = \Psibold \left[ \textstyle{\frac{1}{2m}}, \frac{1}{m} ; \dbold_{\Lcir} \right] ( 0, \Phi_{0, \ave}) 
\end{align*}
on $D_{\Lcir}(1/m)$; and that on $\Sph^2\setminus L_0$, 
\begin{align*}
\Phi_0 = \Ghat_0 +\Phat_0+ \Phip_0, 
\quad
E_0'  = -\Ltilde (\Ghat_0+ \Phat_0) = \Ltilde \Phip_0.
\end{align*}
\end{definition}
\begin{lemma}[Estimates on $\Ghat_0$]
\label{LGhat}
The following hold. 
\begin{enumerate}[label=\emph{(\roman*)}]
\item $\| \Ghat_0 : C^k_{\sym}( \Sph^2 \setminus D_{L_0}(\delta), \gtilde) \| \leq C(k)$. 
\item $\| \Ghat_0 - \log( m \dbold_{L_0}) : C^k(D_{L_0}(\delta) \setminus L_0, \dbold_{L_0}, g,\dbold^2_{L_0}|\log (m \dbold_{L_0})|) \|\leq C(k)$.
\end{enumerate}
\end{lemma}
\begin{proof}
Item (i) follows from Lemma \ref{LG}(iii) and the uniform estimates, in the $\gtilde$ metric, on the cutoff $\Psibold$ in \eqref{EGhat0}.  Next, on $D_{L_0}(\delta)$ we have
\begin{align*}
\Ghat_0 - \log (m \dbold_{L_0}) 
&=
G\circ \dbold_{L_0} -\cos \dbold_{L_0}  \log \dbold_{L_0}
+ \log (m\dbold_{L_0}) ( \cos \dbold_{L_0} - 1).
\end{align*}
The estimate (ii) follows from this, the definitions, and Lemma \ref{LG}(iii).
\end{proof}
\begin{lemma}[Estimates on $E'_0$]
\label{LEp}
\phantom{ab}
\begin{enumerate}[label=\emph{(\roman*)}]
	\item $E'_0$ is supported on $D_{\Lcir}(1/m)\setminus D_{L}(2\delta)$.
	\item  $E'_{0, \osc}$ is supported on $D_{\Lcir}(3\delta)$.
	\item $ \| E'_0 : C^k_{\sym}( D_{\Lcir}(1/m) , \gtilde) \| \leq C(k)$.
	\end{enumerate}
In (iii), the same estimate holds if $E'_0$ is replaced by $E'_{0, \ave}$, or by $E'_{0, \ave}$. 
\end{lemma}
\begin{proof}
The statements on the support of $E'_0$ and $E'_{0, \osc}$ follow from Definition \ref{dGhat}.  Next, note that the variants of (iii) obtained by replacing $E'_0$ by either $E'_{0, \ave}$ or $E'_{0, \osc}$ follow from (iii) by taking averages and subtracting, so it suffices to prove (iii).  To this end,  Definition \ref{dGhat} and Lemma \ref{LPhiave} imply that
\begin{align*}
E'_0 = \Ltilde \Psibold \left[ \textstyle{\frac{1}{2m}}, \frac{1}{m} ; \dbold_{\Lcir} \right] \left( \textstyle{\frac{m}{2}} \sin \dbold_{\Lcir}, 0\right)
\quad
\text{on}
\quad
D_{\Lcir} (1/m) \setminus D_{\Lcir}(3\delta).
\end{align*}
Thus, when restricted to $D_{\Lcir} (1/m) \setminus D_{\Lcir}(3\delta)$, the bound in (i) follows from Lemma \ref{LPhiave} and the uniform bounds on the cutoff in the $\gtilde$ metric.  
It remains to prove the bound in (i) on $D_{\Lcir}(3 \delta)$.  For this, note first that $\Ltilde \Ghat_0$ vanishes on $D_{\Lcir}(2\delta)$ and that $\Ltilde \Phat_0 = 0$ on $D_{\Lcir}(3 \delta)$.  The required bound now follows from the estimates on $\Ghat_0$ in Lemma \ref{LGhat}(i).  
\end{proof}
\begin{lemma}[Estimates on $\Phip_0$ and on $\Phi_0$]
\label{LPhip}
The following hold.
\begin{enumerate}[label=\emph{(\roman*)}]
\item $\| \Phip_{0, \osc} : C^k_{\sym}(\Sph^2, \gtilde)\| \leq C(k)$.
\item $\| \Phip_{0,\osc} : C^k_{\sym}(\Sph^2 \setminus D_{\Lcir}(1/m), \gtilde, (c\, \dbold_{L_2})^{m})\| \leq C(k)$, where $c>0$.
\item $\| \Phip_0 : C^k_{\sym}(\Sph^2, \gtilde) \| \leq C(k)$. 
\item $\| \Phip_0 - \Phip_0(p_0) : C^k_{\sym}(D_{L_0}(\delta), \gtilde,  m^2\dbold^2_{L_0}) \| \leq C(k)$. 
\item $\| \Phi_0- m/2 : C^k_{\sym}( D_{L_2}(\delta), \gtilde, m^2\dbold^2_{L_2}) \| \leq C(k)$.
\end{enumerate}
\end{lemma}
\begin{proof}
For (i), recall from Definition \ref{dGhat} that $\tilde{\Lcal} \Phip_{0,\osc} = E'_{0,\osc}$ on $\Sph^2$.  Because of this, the bounds on $E'_{0, \osc}$ in Lemma \ref{LEp}, and because the quotient of $D_{\Lcir}(1/m)$ by the group $\group$ has uniformly bounded geometry in the $\gtilde$ metric, it follows from standard elliptic theory that
\begin{align}
\label{Ephipap}
\| \Phip_{0, \osc} : C^{2, \alpha}_{\sym}( D_{\Lcir}(1/m), \gtilde)\| < C. 
\end{align}
We next estimate $\Phip_{0, \osc}$ on $\Sph^2 \setminus D_{\Lcir}(1/m) = D_{L_2}(\pi/2- 1/m)$.  First, using \eqref{Ephipap} and a straightforward separation of variables argument in conjunction with the $\group$-symmetry (recall Lemma \ref{Lharm0}), there is a constant $c>0$, independent of $m$, such that
\begin{align}
\label{Ephipap2}
\| \Phip_{0, \osc} \|_{L^2( D_{L_2}(r)) } \leq (cr)^m , 
\quad
r \in (0, \pi/2-1/m ].
\end{align}
Because $\Lcal \Phip_{0, \osc} = 0$ on $\Sph^2 \setminus D_{\Lcir}(3/m)$, items (i) and (ii) now follow from the decay estimate \eqref{Ephipap2} and standard elliptic theory.

Next, we prove (iii).  Because of the estimates on $\Phip_{0, \osc}$ established in (i), it is enough to prove (iii) for $\Phip_{0, \ave}$.  By Definition \ref{dGhat}, it follows that $\Phip_{0, \ave}$ is supported on $D_{\Lcir}(1/m)$, where it satisfies 
\begin{align*}
\Phip_{0,\ave} = 
\begin{cases}
\Psibold \left[ \frac{1}{2m}, \frac{1}{m} ; \dbold_{\Lcir} \right] ( \frac{m}{2} \sin \dbold_{\Lcir}, 0) 
\quad \, \,\, 
\text{on }
D_{\Lcir}(1/m)/ D_{\Lcir}(1/2m), 
\\
\frac{m}{2} \sin \dbold_{\Lcir} - \Ghat_{0, \ave}
\quad \quad \quad   
\quad \quad
\quad 
\text{ on }
D_{\Lcir}(1/2m).
\end{cases}
\end{align*}
We first establish the estimate on $D_{\Lcir}(1/2m)$, where the preceding shows 
\begin{align*}
\Phip_{0,\ave} = \frac{m}{2}\sin \dbold_{\Lcir} - \Ghat_{0, \ave}.
\end{align*}
Note that the left-hand side is smooth, and the derivative jumps on the right-hand side cancel.  We will estimate $\Phip_{0, \ave}$ using that it solves the equation $\Ltilde \Phip_{0, \ave} = E'_{0, \ave}$ on $D_{\Lcir}(1/2m)$, which amounts to an ODE.  More specifically, define a coordinate $\shat$ on $D_{\Lcir}(1/m)$ by requesting that $\tanh (\shat/m) = \sin \xx$, where $\xx$ is the geographic latitude on $\Sph^2$.  With these coordinates, we have (see for example equations (2.12) and (4.21) in \cite{CPAM})
\begin{align*}
\partial^2_{\shat} \Phip_{0,\ave} + 2 m^{-2} \sech^2( \shat/m ) \Phip_{0,\ave} = E'_{0,\ave}. 
\end{align*}
On a neighborhood of $\partial D_{\Lcir}(1/2m)$, we have that $\Ghat_{0,\ave} = 0$ from Definition \ref{dGhat}.  This combined with obvious estimates on $\frac{m}{2} \sin \dbold_{\Lcir}$ implies that $|\Phip_{0,\ave}| < C$ and $|\partial_{\shat} \Phip_{0,\ave}|< C$ on $\partial D_{\Lcir}(1/2m)$. Using this as initial data for the ODE and bounds on $E'_{0,\ave}$ from Lemma \ref{LEp}(ii) complete the proof of (iii) on $D_{\Lcir}(1/2m)$. 
Finally, the estimate on $D_{\Lcir}(1/m)\setminus D_{\Lcir}(1/2m)$ is similar, but even easier since $\Ghat_{0,\ave} = 0$ there, so we omit the details. 

Next, since $\Phip_0 - \Phip_0(p_0)$ has vanishing value and differential at $L_0$ by the symmetries, item (iv) follows from item (iii) and the Taylor expansion for $\Phip_0$ about $L_0$.  
For (v), note first from Definition \ref{dGhat} that
\begin{align*}
\Phi_0 = \frac{m}{2} \sin \dbold_{\Lcir} + \Phip_{0, \osc}
\quad
\text{on} 
\quad
D_{L_2}(1/10).
\end{align*}
The estimate in (v) then follows from combining this decomposition with the estimates on $\Phip_0$ in (ii) above, and using the Taylor expansion for $\frac{m}{2} \sin \dbold_{\Lcir}$ about $L_2$. 
\end{proof}
\subsection{Estimates on $\Phi_2$}
\begin{lemma}
\label{dPhi2}
The function $\Phi_2$ satisfies the following.
\begin{enumerate}[label=\emph{(\roman*)}]
\item $\Phi_2 = 1 + \cos \dbold_{L_2}\log \tan \frac{ \dbold_{L_2}}{2}$.
\item $\| \Phi_2 - \log ( \frac{e}{2} \dbold_{L_2 })  : C^k ( D_{L_2}(1/10), \dbold_{L_2}, g, \dbold_{L_2}^2 |\log \dbold_{L_2}|) \| \leq C(k)$. 
\item $\| \Phi_2 - 1 : C^k( D_{L_0}(\delta), \dbold_{L_0}, g, \dbold_{L_0}^2)\| \leq C(k)$. 
\end{enumerate}
\end{lemma}
\begin{proof}
The right hand side of (i) is $\group$-invariant and can be rewritten as 
\begin{align}
\label{Ephi2g}
G \circ \dbold_{L_2} + (1-\log 2) \cos \dbold_{L_2},
\end{align}
hence by Lemma  \ref{LG} and Remark \ref{rk0} solves the equation $\Lcal u = 0$ on $\Sph^2 \setminus L_2$; further, in combination with its obvious logarithmic behavior near $L_2$, the identity (i) follows from the uniqueness assertion in Lemma \ref{Lldun}. 
The estimate (ii) follows from the expression for $\Phi_2$ in \eqref{Ephi2g} and Lemma \ref{LG}, and (iii) follows from (i) and the definitions, using that $\Phi_2 - 1$ has vanishing value and differential at points in $L_0$. 
\end{proof}
\subsection{Estimates on $\varphi$}
We now introduce the LD solutions which supply the starting point for our constructions.
\begin{lemma}
\label{Lphitau}
There is a  unique LD solution $\varphi$ with the properties that
\begin{enumerate}
\item $\varphi$ is $\group$-symmetric with singular set $L = L[m]$ as in \ref{dL}; 
\item $\varphi$ has vanishing mismatch: $\Mcal_L \varphi = 0$. 
\end{enumerate}
Moreover, the following hold.
\begin{enumerate}[label=\emph{(\roman*)}]
\item For some $c$ independent of $m$, there is $\zeta\in \R$ with $|\zeta | < c$ such that
	\begin{enumerate}[label=\emph{(\alph*)}]
	\item $\log \tau_0 = - \sqrt{m/2} - \frac{3}{4} \log m + \zeta - \Phip_0(p_0)$.
	\item $\tau_2/ \tau_0 = \sqrt{m/2} - \frac{1}{4} \log m - \zeta $.
	\end{enumerate}
\item $\| \varphi - \tau_0 \log (\dbold_{L_0} / \tau_0) : C^k( A_0, \tau^{-2}_0 g) \| \leq 
C(k) m^{2} \tau^3_0$. 
\item $\| \varphi - \tau_2 \log (\dbold_{L_2} / \tau_2) : C^k(A_2, \tau^{-2}_2 g) \| \leq C(k) m^{3/2} \tau^3_2$.  
\end{enumerate}
In (ii) and (iii), the sets $A_i$ are neighborhoods of $D_{L_i} (\tau_i)$ defined by
\begin{align}
\label{EAi}
A_i : = D_{L_i}(2 \tau_i)\setminus D_{L_i} (\tau_i/2), 
\quad
i = 1, 2. 
\end{align}
\end{lemma}
\begin{proof}
\noindent For any LD solution $\varphi$ satisfying (1) with configuration $\tau$, note that
\begin{align}
\label{Evarphi}
\varphi  = \tau_0 \Phi_0 +\tau_2 \Phi_2.
\end{align}
We now examine the matching equations $\Mcal_{i} \varphi = 0$, starting with $\Mcal_0 \varphi$.  Expanding $\varphi$ using \eqref{Evarphi}, expanding $\Phi_0$ using Definition \ref{dGhat} and noting by \ref{dGhat} that $\Phat_0$ vanishes on $D_{L_0 }(1/2m)$, subtracting $\tau_0 \log (\dbold_{L_0} /\tau_0)$ from both sides of the result, and adding and subtracting $\tau_0 \Phip_0(p_0)$ and $\tau_2$, we see that
\begin{equation}
\label{Ephi0ex}
\varphi - \tau_0 \log ( \dbold_{L_0} / \tau_0) = \tau_0 \log ( m e^{\Phip_0(p_0)} \tau_0 ) + \tau_2 + \tau_0(I + II + III)
\end{equation}
on $D_{L_0}(1/2m)$, where the error terms $I$, $II$, and $III$ are defined by
\begin{align*}
I : = \Ghat_0 - \log(m \dbold_{L_0}), 
\quad
II : = \Phip_0 - \Phip_0(p_0), 
\quad
III : = \tau_2/\tau_0 ( \Phi_2 - 1).
\end{align*}
By Lemma \ref{LGhat}(ii), Lemma \ref{LPhip}(iv), and \ref{dPhi2}(iii), the error terms $I, II$, and $III$ have vanishing value and differential at $L_0$ and satisfy the estimates
\begin{align*}
\| I \| \leq C(k) \tau_0^2 \log (m \tau_0), 
\quad
\| II \| \leq C(k) m^{2} \tau^2_0, 
\quad
\| III \| \leq C(k)\sqrt{m} \tau_0^2,
\end{align*}
where each norm is the $C^k(A_0, \tau^{-2}_0 g)$ norm.  The condition $\Mcal_0 \varphi = 0$ is thus equivalent to 
\begin{align}
\label{Ematch0}
 \tau_0 \log (m e^{\Phip_0(p_0)} \tau_0) + \tau_2 = 0;
\end{align}
moreover, if $\Mcal_0 \varphi = 0$, so that \eqref{Ematch0} is satisfied, then the estimates above imply (ii).
To study $\Mcal_{2} \varphi$, we
expand $\varphi$ using \eqref{Evarphi}, subtract $\tau_2 \log (\dbold_{L_2} / \tau_2)$ from both sides of the result, and add and subtract  $\tau_0 \frac{m}{2}$ and $\tau_2 \log (e/2) $ to the right hand side to see that
\begin{equation*}
\begin{gathered}
\varphi -\tau_2 \log (\dbold_{L_2}/ \tau_2)=\tau_0 \frac{m}{2} +  \tau_2\log( \frac{e}{2} \tau_2)+ \tau_2( IV + V)
\end{gathered}
\end{equation*}
where the error terms $IV$ and $V$ are defined by
\begin{align*}
IV := \frac{\tau_0}{\tau_2}(\Phi_0- m/2),
\quad
V:= \Phi_2 - \log( \frac{e}{2} \dbold_{L_2}).
\end{align*}
By Lemma \ref{LPhip}(v) and Lemma \ref{dPhi2}(ii), the error terms $IV$ and $V$ have vanishing value and differential at $L_2$ and satisfy the estimates
\begin{align*}
\| IV\| \leq C(k) m^{3/2} \tau^2_2,
\quad
\| V\| \leq C(k) \tau_2^2 |\log \tau_2|,
\end{align*}
where each norm is the $C^k(A_2, \tau^{-2}_2 g)$ norm.  The condition $\Mcal_2 \varphi = 0$ is thus equivalent to
\begin{align}
\label{Ematch2}
\tau_0 \frac{m}{2} + \tau_2 \log ( \frac{e}{2} \tau_2) = 0;
\end{align}
moreover, if $\Mcal_2 \varphi = 0$, so that \eqref{Ematch2} is satisfied, then the estimates above imply (iii).

Now suppose $\varphi$ has vanishing mismatch.  By combining \eqref{Ematch0} and \eqref{Ematch2}, we see the ratio $r: = \tau_2 / \tau_0$ satisfies
\begin{align}
\label{Er}
\frac{m}{2} - r^2 + r \log \left( \frac{e}{2m} e^{- \Phip_0(p_0)} r\right) = 0.
\end{align}
By elementary calculus, \eqref{Er} has a unique positive solution $r$, and \eqref{Ematch0} and \eqref{Ematch2} show $r$ determines $\tau_0$ and $\tau_2$ uniquely, hence determines $\varphi$ uniquely by Lemma \ref{Lldun}.
It remains to prove (i).  For this, observe (i)(b) must be satisfied for some number $\zeta$ (depending on $m$), so that $r = \sqrt{m/2} - \frac{1}{4} \log m - \zeta$.  Substituting this into \eqref{Er} and cancelling some terms reveals that 
\begin{align*}
0 = \sqrt{\frac{m}{2}} \left( 2\zeta + 1- \Phip_0(p_0) - \frac{3}{2} \log 2 + O(\log m/ \sqrt{m})\right) 
\\
+ \left(\frac{1}{4}\log m + \zeta\right)^2
- \left(\frac{1}{4} \log m + \zeta\right) \log \left( \frac{e}{2m}e^{\Phip_0(p_0)} r\right).
\end{align*}
In particular, this implies
\begin{align}
2\zeta =-1 + \Phip_0(p_0) + \frac{3}{2} \log 2 + O((\log m)^2/ \sqrt{m}); 
\end{align}
with the bounds on $\Phip_0$ from Lemma \ref{LPhip}(iii), this implies the bound on $\zeta$.  Finally, (i)(a) follows by substituting the value for $r$ into \eqref{Ematch0}. 
\end{proof}
\section{Spectral Geometry of $\Sph^2 \setminus D$}
\label{Slin}
\subsection{Preliminary notation}
We now introduce some abbreviated notation for the norms we will be using. 
\begin{definition}
\label{dnorms}
For $k \in \N, \alpha \in (0, 1)$, and $\Omega \subset \Sph^2 \setminus L$ a submanifold,  define
\begin{align*}
\| u\|_{k, \alpha; \Omega} : = \| u : C^{k, \alpha}(\Omega, \dbold_L,  g) \|.
\end{align*}
\end{definition}

\begin{convention}
From now on, we fix some $\alpha \in (0, 1)$ for use in H{\"o}lder norms.  We will suppress the dependence of various constants on $\alpha$. 
\end{convention}

\begin{convention}
Whenever $U, V$ are submanifolds of $\Sph^2 \setminus L$ and $X, Y$ are subspaces of $C^{k, \alpha}(U)$ and $C^{j, \alpha}(V)$ respectively, by a \emph{bounded linear map} $T : X \rightarrow Y$ we mean $T$ is linear and its operator norm, computed with respect with respect to the $\| \cdot \|_{k, \alpha; U}$ norm on the domain and the $\| \cdot \|_{j, \alpha; V}$ norm on the target, is bounded by a constant $C$ independent of $m$. 
\end{convention}

\begin{definition}
\label{dAD}
Define domains $D$ and $A$ of $\Sph^2$ by
\begin{align*}
D : = \cup_{i \in \{0, 2\}} D_{L_i} (\tau_i), 
\quad
A: = \cup_{i \in \{0, 2\}}A_i,
\end{align*}
where the domains $D_i$ and $A_i$ are defined by 
\begin{align*}
D_i := D_{L_i}(\tau_i),
\quad
A_i : = D_{L_i}(2 \tau_i) \setminus D_{L_i}(\tau_i/2),
\quad
i=0,2.
\end{align*}
Let $\iota : \partial D \rightarrow \Sph^2$ denote the inclusion.  Define also a locally constant function $\tau$ on $A$ by requesting that $\tau|_{A_i} = \tau_i$, and define a metric $\ghat$ on $A$ by $\ghat = \tau^{-2} g$.  For convenience, denote by $r: = \dbold_L$ on $A$, and given any function $u : A \rightarrow \R$ denote by $\uhat$ the function $\tau^{-1} u$, so that in particular $\rhat = \tau^{-1} \dbold_L$. Finally, let $\nuhat$ denote the $\ghat$-unit outward pointing normal to $\Sph^2 \setminus D$ along $\partial D$, so that $\nuhat = -\partial_{\rhat}|_{\partial D}$. 
\end{definition}

We now recast the estimates on the LD solution $\varphi$ in Lemma \ref{Lphitau}(ii)-(iii) in terms of the H{\"o}lder norms just defined, as follows.
\begin{lemma}
\label{Lphierr}
With $\varphi$ the LD solution from Lemma \ref{Lphitau}, the function 
\begin{align}
\label{Ephierr}
\phierr : = \varphi - \tau \log (\dbold_L / \tau) \in C^{\infty}_{\sym}(A)
\end{align}
satisfies $\| \phierr \|_{3, \alpha; A} \leq \tau^{5/2}_2$.
\end{lemma}
\begin{proof}
This is straightforward from Lemma \ref{Lphitau}, using Definition \ref{dnorms}.
\end{proof}

\begin{notation}
\label{Nhigh}
Subscripts ``low" and ``high" will be used to denote subspaces of functions defined on $\partial D$ which respectively belong to or are $L^2(\partial D)$-orthogonal to the locally constant functions on $\partial D$. 
Given $v \in L^2_{\sym}(\partial D)$, we thus have a unique decomposition
\begin{align*}
v = v_\low + v_\high 
\quad
\text{with}
\quad
v_\low \in L^2_{\sym, \low}(\partial D), 
\quad
v_\high \in L^2_{\sym, \high}(\partial D).
\end{align*}

We also use subscripts ``avg" and ``osc" to denote subspaces of functions defined on $\partial D$ which respectively belong to or are $L^2(\partial D)$-orthogonal to the constant functions on $\partial D$. Given $v \in L^2_{\sym}(\partial D)$, we also have a unique decomposition
\begin{align*}
v = v_\ave + v_\osc, 
\quad
\text{with}
\quad
v_\ave \in L^2_{\sym, \ave}(\partial D) , 
\quad
v_\osc \in L^2_{\sym, \osc}(\partial D). 
\end{align*}
Note in particular that $v_\ave = \frac{1}{|\partial D|} \int_{\partial D} v$.  Finally, denote by $P_\ave, P_\osc$ the projections defined by $P_\ave v= v_\ave$ and $P_\osc v = v_\osc$ .
\end{notation}

Although the definition of $v_\ave$ in Notation \ref{Nhigh} conflicts with a different definition of average from \ref{dave}, no confusion will result, since the application of Definition \ref{dave} is limited to Section \ref{SLD}.

\begin{lemma}
\label{LHDdecomp}
The following hold. 
\begin{enumerate}[label=\emph{(\roman*)}]
\item $L^2_{\sym, \low}(\partial D) = \Hcal^0_{\sym}(\partial D) =  \Hcal^{0}_{\sym, \ave}(\partial D) \oplus \Hcal^{0}_{\sym, \osc}(\partial D)$.
\item $L^2_{\sym, \high}(\partial D) = \big[ \bigoplus_{k \in I_0} \Hcal^k_{\sym}(\partial D_0) \big]
		\oplus \big[ \bigoplus_{k\in I_2} \Hcal^k_{\sym}(\partial D_2) \big]$.
\end{enumerate}
\end{lemma}
\begin{proof}
This follows easily from Definition \ref{dharm} and Lemma \ref{Lharm0}. 
\end{proof}

\begin{lemma}
\label{Ldtnprelim} 
There is a constant $C> 0$ such that for each $v \in \Hcal^{0}_{\sym, \osc}(\partial D)$, 
	\begin{align*}
	\left| 1 - \frac{v|_{\partial D_2}}{v|_{\partial D_0}}\frac{1}{\Phi_2|_{\partial D_2}} \right| \leq \frac{C}{\sqrt{m}}.	
	\end{align*}
\end{lemma}
\begin{proof}
The fact that $v_\ave = \frac{1}{|\partial D|} \int_{\partial D} v = 0$ is equivalent to
\begin{align*}
m (\sin \tau_0) v|_{\partial D_0} + 2 (\sin \tau_2) v|_{\partial D_2} =0;
\end{align*}
after rearranging, expanding the $\sin \tau_i$ terms using Lemma \ref{Lphitau}(i), we find
\begin{align}
\label{Eavg0w}
\frac{v|_{\partial D_2}}{v|_{\partial D_0}}  = - \sqrt{\frac{m}{2}} \bigg( 1 + \frac{1}{4} \frac{\log m}{\sqrt{m/2}} + O \bigg( \frac{1}{\sqrt{m}}\bigg) \bigg).
\end{align}
On the other hand, using the estimate in Lemma \ref{dPhi2}(ii) on $\partial D_2$, and using the expression for $\tau_2$ from Lemma \ref{Lphitau}(i) and simplifying, we have
\begin{align}
\label{Eavg0w2}
\Phi_2|_{\partial D_{2}} = - \sqrt{\frac{m}{2}} \bigg( 1+\frac{1}{4} \frac{\log m}{\sqrt{m/2}} + O\bigg( \frac{1}{\sqrt{m}}\bigg)\bigg).
\end{align}
The conclusion now follows by combining \eqref{Eavg0w} and \eqref{Eavg0w2}.
\end{proof}

\begin{definition}
\label{dgcirc}
Let $\gcir$ be the $\group$-symmetric Riemannian metric defined on $D_L(1/m)$ whose restriction to each component is induced from a system of polar normal coordinates centered at the corresponding point of $L$.
\end{definition}

\subsection{The $\Lcal$-extension operator $H_\Lcal$}
This subsection is concerned with the problem of extending a function $v$ defined on $\partial D$ to a function $H_\Lcal v$ on $\Sph^2 \setminus D$ satisfying $\Lcal H_\Lcal v = 0$. 
As a first step, the next Lemma uses separation of variables to produce an extension $\Hdelta v$ of a given $v \in C^{2, \alpha}_{\sym, \high}(\partial D)$ which is harmonic with respect to $\gcir$ on $D_L(1/m) \setminus D$.

\begin{lemma}
\label{Lext0}
There is a bounded linear map
\begin{align*}
\Hdelta : C^{2, \alpha}_{\sym, \high}(\partial D) \rightarrow C^{2, \alpha}_{\sym}(D_L(1/m) \setminus D)
\end{align*}
with the properties that 
\begin{enumerate}[label=\emph{(\roman*)}]
\item $\Hdelta$ restricts to the identity on $\partial D$.
\item $\Delta_{\gcir} \Hdelta v = 0$ on $ D_L(1/m) \setminus D$. 
\item $\| \Hdelta v : C^{2, \alpha}_{\sym}(D_L(1/m) \setminus D, \dbold_L, g, \rhat^{-2})\| \leq C \| v \|_{2, \alpha; \partial D}$.
\item $\partialnu \Hdelta - 1$ has a bounded inverse
	\begin{align*}
		(\partialnu \Hdelta- 1)^{-1} : C^{1, \alpha}_{\sym, \high}(\partial D) \rightarrow C^{2, \alpha}_{\sym, \high}(\partial D).
	\end{align*}
\end{enumerate}
\end{lemma}
\begin{proof}
By definition \ref{dgcirc}, $D_L(1/m) \setminus D$ is isometric to a disjoint union of Euclidean annuli, each component of which of which can be identified with an annulus $\mathring{A} = \mathring{D}(r_1)\setminus \mathring{D}(r_2)$ in $\R^2$; in particular, if $v \in C^{2, \alpha}_{\sym, \high}(\partial D)$, then $v$ can be considered as defined on each $\partial \mathring{D}(r_2)$.  By the assumption that $v_\low = 0$, there is a unique harmonic function on $\R^2 \setminus \mathring{D}(r_2)$ vanishing at infinity and restricting to $v$ on $\partial \mathring{D}(r_2)$.  We define $\Hdelta v$ to be the function on $D_L(1/m) \setminus D$ arising from these harmonic extensions.

Items (i) and (ii) follow from the definition of $\Hdelta$, and (iii) follows from separation of variables and standard theory (see, for example \cite[Lemma 4.1]{SicbaldiOut}), using the triviality of $\Hcal^1_{\sym}(\partial D)$ (recall Lemma \ref{Lharm0}).

For (iv), recall that $\partial D = \partial D_0 \cup \partial D_2$, where $(\partial D_i, \ghat)$ is a disjoint union of round circles on which $\group$ acts transitively, each circle having length $\sin \tau_i/ \tau_i$ in the metric $\ghat$.  This implies that
\begin{align}
\label{Eeigen00}
	\partialnu \Hdelta v_{k, i} = k \frac{\tau_i}{\sin \tau_i} v_{k, i}
\end{align}
whenever $v_{k, i} \in \Hcal^k(\partial D_i)$ for $i \in \{0, 2\}$ and $k \in I_i$.

From the symmetries and \eqref{Eeigen00}, it follows that the smallest eigenvalue of $\partialnu \Hdelta$ is near $2$; by standard theory (see for example Theorem 7.3 and Remark 2 on p.669 of \cite{Agmon}), this implies $\partialnu \Hdelta-1$ maps $C^{2, \alpha}_{\sym, \high}(\partial D)$ surjectively onto $C^{1, \alpha}_{\sym, \high}(\partial D)$, and that $\|v\|_{2, \alpha; \partial D} \leq C \| (\partialnu \Hdelta - 1) v \|_{1, \alpha; \partial D}$ whenever $v \in C^{2, \alpha}_{\sym, \high}(\partial D)$, implying (iv).
\end{proof}

The next lemma modifies $\Hdelta$ to produce the desired extension operator $H_\Lcal$ satisfying $\Lcal H_\Lcal v = 0$ on  $\Sph^2 \setminus D$; a cost is that $H_\Lcal v$ is only approximately equal to $v$ on $\partial D$, although the difference is a constant whose size is small in terms of $v$.  
Additionally, $H_\Lcal$ is defined on $C^{2, \alpha}_{\sym, \osc}(\partial D)$, as opposed to $C^{2, \alpha}_{\sym, \high}(\partial D)$ for $\Hdelta$.

\begin{prop}
\label{Pext1}
There is a bounded linear map
\begin{align*}
H_{\Lcal} : C^{2, \alpha}_{\sym, \osc}(\partial D) \rightarrow C^{2, \alpha}_{\sym}(\Sph^2 \setminus D)
\end{align*}
such that if $v \in C^{2, \alpha}_{\sym, \osc}(\partial D)$ and $u:= H_{\Lcal} v$, then the following hold.
\begin{enumerate}[label=\emph{(\roman*)}]
\item $\Lcal u  = 0$ on $\Sph^2 \setminus D$.
\item $(u|_{\partial D})_{\osc} = v$ and $|(u|_{\partial D})_{\ave}| \leq \frac{C}{\sqrt{m}} \| v \|_{2, \alpha; \partial D}$.
\item $\| \partialnu H_\Lcal v -\partialnu \Hdelta v_\high \|_{1, \alpha; \partial D} \leq \frac{C}{\sqrt{m} } \| v \|_{2, \alpha; \partial D}$.
\end{enumerate}
\end{prop}
\begin{proof}
The proof is split into five steps. \newline
\emph{Step 1: the approximate extension operator $\Happr$}.  Here we define a linear map
\begin{align*}
\Happr : C^{2, \alpha}_{\sym, \osc}(\partial D) \rightarrow C^{2, \alpha}_{\sym}(\Sph^2 \setminus D)
\end{align*}
which will be the basis for the definition of $H_\Lcal$ in Step 5.  In view of Notation \ref{Nhigh} and Lemma \ref{LHDdecomp}, we have
\begin{align*}
C^{2, \alpha}_{\sym, \osc}(\partial D) = \Hcal^0_{\sym, \osc}(\partial D) \oplus C^{2, \alpha}_{\sym, \high}(\partial D) 
\end{align*}
and then define $\Happr$ as a direct sum of maps
\begin{align*}
\Hlow :  \Hcal^0_{\sym, \osc}(\partial D) \rightarrow C^{2, \alpha}_{\sym}(\Sph^2 \setminus D), 
\\
\Hhigh : C^{2, \alpha}_{\sym, \high}(\partial D) \rightarrow C^{2, \alpha}_{\sym}(\Sph^2 \setminus D),
\end{align*}
where $\Hlow$ is defined by 
\begin{align}
\label{Ehlow}
\Hlow v = \frac{v |_{\partial D_2}}{\Phi_2|_{\partial D_2}}\Phi_2
\end{align}
and $\Hhigh: = \Hdeltatilde + \Hdeltatildep$, where $\Hdeltatilde$ and $\Hdeltatildep$ are defined by requesting that
\begin{equation}
\label{Ehpp}
\begin{gathered}
\mathrm{supp}\,  \Hdeltatilde v \subset D_L(1/m) \setminus D \quad \text{and} 
\\
\Hdeltatilde v = \Psibold\left[ \textstyle{ \frac{1}{2m}}, \frac{1}{m}; \dbold_{L} \right] \left( \Hdelta v,  0\right)
\quad
\text{on}
\quad
D_{L}(1/m) \setminus D,
\end{gathered}
\end{equation}
\begin{equation}
\begin{gathered}
\label{Ehtp}
\Lcal \Hdeltatildep v = - E
\quad
\text{on}
\quad 
\Sph^2, 
\quad
\text{where}
\\
E = \Lcal \Hdeltatilde v \text{ on } \Sph^2 \setminus D,  \quad \Delta E = 0 \text{ on } D, \quad E|_{\partial D} = (\Lcal \Hdeltatilde v)|_{\partial D}
\end{gathered}
\end{equation}
whenever $v \in C^{2, \alpha}_{\sym, \high}(\partial D)$. 
In particular, notice that $E$ is a well-defined element of $C^{0, \alpha}_{\sym}(\Sph^2)$ due to the smallness of $D$ and the choice of boundary values, and that the existence, uniqueness, and linearity of $\Hdeltatildep$ is ensured by the Fredholm alternative and standard linear theory.

\emph{Step 2: estimates for $\Hlow$.  }Since $\Lcal \Phi_2 = 0$ on $\Sph^2 \setminus L_2$, we have that 
\[ 
\| \Phi_2 / \Phi_2|_{\partial D_2} \|_{2, \alpha; \Sph^2 \setminus D} \leq C
\]  by standard elliptic estimates, and clearly $| v|_{\partial D_2}| \leq \| v \|_{2, \alpha; \partial D}$. 
It follows that $\Hlow$ is a bounded linear map.

We next estimate estimate $\partialnu \Hlow v$.  From Lemma \ref{dPhi2}(ii)-(iii), we have
\begin{align*}
\| \nuhat \Phi_2\|_{1, \alpha; \partial D_0 } \leq C \tau^2_0, 
\quad
\| \nuhat \Phi_2\|_{1, \alpha; \partial D_2} \leq C,
\end{align*}
and from \eqref{Eavg0w2} that $| \Phi|_{\partial D_2} | \geq C \sqrt{m}$. 
Combining these estimates with \eqref{Ehlow} shows that
\begin{align}
\label{EElow2}
\| \partialnu \Hlow v \|_{1, \alpha; \partial D} \leq  \frac{C}{\sqrt{m}} \| v\|_{1, \alpha; \partial D},
\end{align}
and a short calculation shows the boundary values satisfy
\begin{equation*}
\begin{aligned}
\Hlow v - v &= 0
\quad
\text{on} 
\quad
\partial D_2, 
\\
\Hlow v - v &= \left( \frac{v|_{\partial D_2}}{v|_{\partial D_0}} \frac{1}{\Phi_2|_{\partial D_2}} - 1\right) \Phi_2 v
\quad
\text{on} 
\quad
\partial D_0.
\end{aligned}
\end{equation*}
Estimating  this using Lemma \ref{Ldtnprelim} and Lemma \ref{dPhi2}(iii) shows that
\begin{align}
\label{Ehapperlow}
\| \Hlow v - v\|_{2, \alpha; \partial D} \leq \frac{C}{\sqrt{m}} \| v\|_{2, \alpha; \partial D}.
\end{align}
By \eqref{Ehlow} and the fact that $\Lcal \Phi_2 = 0$, observe also that $\Lcal \Hlow v = 0$. 

\emph{Step 3: estimates for $\Hhigh$.  }First, note from \eqref{Ehpp} and \eqref{Ehtp} that
\begin{align}
\label{Eutp1}
\begin{cases}
\Lcal \Hhigh v  = 0 \hfill \quad &\text{in} \quad \Sph^2 \setminus D\\
\phantom{\Lcal} \Hhigh v= v+ \Hdeltatildep v \hfill \quad &\text{on} \quad \partial D.
\end{cases}
\end{align}

We next estimate $\Hdeltatildep v$.  For ease of notation, in what follows we denote 
\[ 
\Omega : = D_L(1/2m) \setminus D. 
\] 
We first estimate $E$ on $\Omega$, where it satisfies $E = \Lcal H_\Delta v= (\Delta +2) H_\Delta v$.  Using \eqref{Ehpp} and that $\Delta_{\gcir} \Hdelta v = 0$ from Lemma \ref{Lext0}, we estimate
\begin{align*}
\| E\|_{0, \alpha; \Omega} \leq C \| (\Delta_g - \Delta_{\gcir}) \Hdelta v \|_{0, \alpha; \Omega} + C\| \Hdelta v\|_{0, \alpha; \Omega}.
\end{align*}
Estimating the difference of the Laplacians by a direct calculation, or by using  \cite[Lemma C.10]{LDG} and \cite[Lemma 2.22(iv)]{LDG}, shows that
\begin{align*}
\| (\Delta_g - \Delta_{\gcir} ) H_\Delta v\|_{0, \alpha;  \Omega} \leq C \| H_\Delta v \|_{2, \alpha; \Omega} \leq C \| v \|_{2, \alpha; \partial D},
\end{align*}
where we have used Lemma \ref{Lext0}(iii).  Consequently
\begin{align}
\label{Eharmd}
\| E \|_{0, \alpha; \Omega} \leq C \| v \|_{2, \alpha; \partial D}, 
\end{align}
and using this with the the definition of $E$ on $D$ from \eqref{Ehtp} implies
\begin{align*}
\| E\|_{L^\infty(D) } = \| E \|_{L^\infty( \partial D) } \leq C \| v \|_{2, \alpha; \partial D}. 
\end{align*}

In similar fashion, using \eqref{Ehpp} and the definitions, we have
\begin{align*}
\| E\|_{0, \alpha; D_L(1/m) \setminus D_L(1/2m)} \leq Cm^2 \| H_\Delta v \|_{2, \alpha;  D_L(1/m) \setminus D_L(1/2m)} \leq C \| v \|_{2, \alpha; \partial D},
\end{align*}
where the second inequality uses the decay estimate in Lemma \ref{Lext0}(iii).

In total, we have $\| E \|_{L^\infty(\Sph^2)} \leq C \| v \|_{2, \alpha; \partial D}$, and furthermore
\begin{align*}
\| E \|_{L^2(\Sph^2)} \leq \| E\|_{L^\infty(\Sph^2)} | D_L(1/m)  |^{1/2} \leq \frac{C}{\sqrt{m}} \| v \|_{2, \alpha; \partial D},
\end{align*}
where the last inequality estimates the area of $D_L(1/m)$, which consists of $m+2$ disks, each with radius $1/m$.

Recalling that $\Lcal \Hdeltatildep v = - E$ from \eqref{Ehtp} and combining the preceding with standard elliptic theory and De Giorgi-Nash-Moser theory implies that
\begin{align*}
\| \Hdeltatildep v\|_{C^0(\Sph^2)} 
\leq C \| \Hdeltatildep v \|_{L^2(\Sph^2)}
\leq C \| E \|_{L^2(\Sph^2)} 
\leq \frac{C}{\sqrt{m}} \| v \|_{2, \alpha; \partial D}.
\end{align*}

We now obtain $C^{2, \alpha}$ estimates on $\Hdeltatildep v$, first on $A$.  Because $A$ has uniformly bounded geometry with respect to $\ghat$ and $\tau^2 \Lcal \Hdeltatildep v = \tau^2 E$, Schauder theory and the above estimates imply
\begin{align*}
\| \Hdeltatildep v \|_{2, \alpha; A} &\leq C ( \| \Hdeltatildep v \|_{C^0(A)} + \| \tau^2 E \|_{0, \alpha; A})
\leq \frac{C}{\sqrt{m}} \| v \|_{2, \alpha; \partial D}.
\end{align*}
Arguing analogously for other subdomains and combining shows that
\begin{align}
\label{Etpv}
\| \Hdeltatildep v \|_{2, \alpha; \Sph^2 \setminus D} \leq \frac{C}{\sqrt{m}} \| v \|_{2, \alpha; \partial D}.
\end{align}

We now collect our estimates for $\Hhigh$: from \eqref{Eutp1} and \eqref{Etpv}, we have
\begin{align}
\label{Ehapperup}
\| \Hhigh v - v \|_{2, \alpha; \partial D} \leq \frac{C}{\sqrt{m}} \| v \|_{2, \alpha; \partial D}, 
\end{align}
and combining \eqref{Ehpp}, Lemma \ref{Lext0}, and  \eqref{Etpv} proves $\Hhigh$ is bounded.

\emph{Step 4: Estimates for $\Happr$.  } Because $\Happr = \Hlow + \Hhigh$, the preceding shows that $\Happr$ is bounded and satisfies $\Lcal \Happr v = 0$ on $\Sph^2 \setminus D$ whenever $v \in C^{2, \alpha}_{\sym, \osc}(\partial D)$.  Also, combining \eqref{Ehapperlow} and \eqref{Ehapperup} shows that
\begin{align}
\label{Ehhigh00}
\| \Happr v - v\|_{2, \alpha; \partial D} \leq  \frac{C}{\sqrt{m}} \|v\|_{2, \alpha; \partial D}.
\end{align}
Finally, from the definitions,  \eqref{EElow2}, and \eqref{Etpv}, we have
\begin{equation}
\begin{aligned}
\label{Eh00}
\| \partialnu \Happr v - \partialnu \Hdelta v_\high\|_{1, \alpha; \partial D} &= 
\| \partialnu \Hlow v_\low+\partialnu \Hdeltatildep v_{\high} \|_{1, \alpha; \partial D} 
\\
&\leq \frac{C}{\sqrt{m}} \| v \|_{2, \alpha; \partial D}. 
\end{aligned}
\end{equation}

\emph{Step 5: the exact extension operator.} By \eqref{Ehhigh00}, the
map
\begin{align*}
 P_{\osc}  \iota^*  \Happr - 1: C^{2,\alpha}_{\sym, \osc}(\partial D) \rightarrow C^{2, \alpha}_{\sym, \osc}(\partial D)
\end{align*}
given by $v \mapsto (\iota^*\Happr v)_{\osc} - v$ (recall Definition \ref{dAD} and Notation \ref{Nhigh}) has operator norm bounded by $C /\sqrt{m}$; thus $P_{\osc}  \iota^*  \Happr $ has an inverse with uniformly bounded norm.  We then define
\begin{align*}
H_\Lcal = \Happr (P_{\osc}  \iota^*  \Happr )^{-1};
\end{align*}
item (i) follows immediately from this and the fact, established above, that $\Lcal \Happr v =0$.  The definition of $H_\Lcal$ also implies the first part of (ii), and that  \begin{align}
\label{EHLdiff}
\| (H_\Lcal -\Happr) v \|_{2, \alpha; \Sph^2 \setminus D} \leq  \frac{C}{\sqrt{m} } \| v \|_{2, \alpha; \partial D}.
\end{align}
Combined with \eqref{Eh00}, the remaining parts of the proposition now follow.
\end{proof}

While Proposition \ref{Pext1} provides a solution of the homogeneous equation $\Lcal u = 0$ on $\Sph^2 \setminus D$ with \emph{prescribed} oscillatory part on $\partial D$,  the next Proposition provides a solution of the inhomogeneous equation $\Lcal u = E$ on $\Sph^2 \setminus D$ with \emph{zero} oscillatory part on $\partial D$.
\begin{prop}
\label{Pinhomog}
There is a bounded linear map
\begin{align*}
J_\Lcal :  \left\{ E \in C^{0, \alpha}_{\sym}(\Sph^2 \setminus D) : \mathrm{supp} \, E \subset A \setminus D\right\} 
\rightarrow C^{2, \alpha}_{\sym}(\Sph^2 \setminus D)
\end{align*}
such that if $E$ is in the domain of $J_{\Lcal}$ and $u= J_{\Lcal} E$, then the following hold. 
\begin{enumerate}[label=\emph{(\roman*)}]
\item $\Lcalhat u = E$ on $\Sph^2 \setminus D$. 
\item $(u|_{\partial D})_\osc = 0$ and  $|(u|_{\partial D})_{\ave}|   \leq C \| E \|_{0, \alpha; A \setminus D}$.

\end{enumerate}
\end{prop}
\begin{proof}
We first decompose $E = E_\low + E_\high$, where $E_\low$ is constant on each circle $\partial D_{L_i} (r)$ for $i \in \{0, 2\}$ and $r \in (\tau_i, 2 \tau_i)$ and define $J_\Lcal E_\low$ and $J_\Lcal E_\high$ separately.

There is a unique ODE solution  $u_\low \in C^{2, \alpha}_{\sym}(A \setminus D)$ depending only on $\dbold_{L_i}$ on each $A_i \setminus D_i$ solving $\Lcalhat u_\low =E_\low$ with the initial conditions
\begin{align*}
u|_{\partial D_{L_i}} ( 2\tau_i) = 0, 
\quad
\partial_{r} u |_{\partial D_{L_i}}(2 \tau_i) = 0
\quad
i \in \{0, 2\}. 
\end{align*}
In particular, these conditions and the assumption $\mathrm{supp} E \subset A \setminus D$ implies $u_\low$ can be considered smooth on all $\Sph^2 \setminus D$ and supported on $A \setminus D$, and basic ODE theory implies
\begin{align}
\label{Euode}
\| u_\low \|_{2, \alpha; \Sph^2\setminus D} \leq C \| E_\low \|_{0, \alpha; A \setminus D}. 
\end{align}

By standard theory and separation of variables (for example, see \cite[Proposition 5.13]{Wiygul}), there is a $C^{2,\alpha}$ function $u_{\high}$ defined on $D_L(1/m) \setminus D$, solving $\Lcalhat u_{\high} = E_\high$ on $D_L(1/m)\setminus D$ with the estimate
\begin{align}
\label{Eupest1}
\| u_{\high}:  C^{2, \alpha}_{\sym}(D_L(1/m) \setminus D, r, g, \rhat^{-2})\|
\leq 
C 
\| E_\high \|_{0, \alpha; A \setminus D}.
\end{align}

We then define $\Jtilde_\Lcal E = u_\low +  \utilde_{\high} + \utilde^{err}_{\high}$, where $\utilde_{\high}, \utilde_{\high}^{err} \in C^{2, \alpha}_{\sym}(\Sph^2 \setminus D)$ are defined by requesting that
\begin{equation*}
\begin{gathered}
\mathrm{supp} \, \utilde_{\high} \subset D_L(1/m) \setminus D 
\quad
\text{and}
\\
\utilde_{\high} =  \Psibold\left[ \textstyle{ \frac{1}{2m}}, \frac{1}{m}; \dbold_{L} \right]  \left( u_\high,  0\right)
\quad
\text{on} 
\quad
D_L(1/m) \setminus D, 
\end{gathered}
\end{equation*}
\begin{equation*}
\begin{gathered}
\Lcal \utilde^{err}_{\high} = - \Etilde 
\quad
\text{on}
\quad \Sph^2
\quad \text{where}
\\
\mathrm{supp} \, \Etilde \subset D_L(1/m) \setminus D_L(1/2m),
\quad \Etilde = \Lcal u_{\high}.
\end{gathered}
\end{equation*}
From this definition, it follows that $\Lcal \Jtilde_\Lcal E = E$ on $\Sph^2 \setminus D$.  Moreover,  \eqref{Euode}, \eqref{Eupest1}, and the bound
\begin{align*}
\| \utilde^{err}_{\high} \|_{2, \alpha; \Sph^2 \setminus D} \leq \| E\|_{0, \alpha; A \setminus D}
\end{align*}
which follows by arguing as in the proof of \ref{Pext1} show that $\Jtilde_\Lcal$ is bounded. 

Finally, we define 
\begin{align}
\label{EJdef}
J_{\Lcal} E := \Jtilde_\Lcal E  - H_{\Lcal} (( \Jtilde_\Lcal E)|_{\partial D})_\osc
\end{align}
and observe using the preceding that $J_\Lcal$ satisfies the desired properties.
\end{proof}

\begin{prop}
\label{Plin}
The map $\Bcal : C^{2, \alpha}_{\sym, \osc}(\partial D) \rightarrow C^{1, \alpha}_{\sym, \osc}(\partial D)$ defined by
\begin{align}
\label{dBcal}
\Bcal v = (\hat{\nu} H_\Lcal v - v)_{\osc}
\end{align} 
has a bounded right inverse $\Rcal :  C^{1, \alpha}_{\sym, \osc}(\partial D)  \rightarrow C^{2, \alpha}_{\sym, \osc}(\partial D)$; that is $\Bcal \Rcal$ is the identity map on $C^{1, \alpha}_{\sym, \osc}(\partial D)$.
\end{prop}
\begin{proof}
We first define an approximate solution operator
\begin{equation*}
\begin{aligned}
&\Rcalappr : C^{1, \alpha}_{\sym, \osc}(\partial D) \rightarrow C^{2, \alpha}_{\sym, \osc}(\partial D) \quad \text{by}
\\
&\Rcalappr E =  (\partialnu \Hdelta -1)^{-1} E_\high - E_\low, 
\end{aligned}
\end{equation*}
where $E: = E_\low + E_\high$ with $E_\low \in \Hcal^0_{\sym}(\partial D)$ and $E_{\high} \in C^{1,\alpha}_{\sym, \high}(\partial D)$.

Now let $E \in C^{1, \alpha}_{\sym, \osc}(\partial D)$ be given, and set $v := \Rcalappr E$.  We subdivide the remainder of the proof into three steps: 

\emph{Step 1: Estimates on $\Rcalappr$.}
The condition that $E_\low \in \Hcal^0_{\sym}(\partial D)$ means that $E_\low$ is locally constant, and the fact that $E$ has average zero implies further that   $E_\low$ is equal to $\frac{1}{|\partial D_i|} \int_{\partial D_i} E$ on $\partial D_i$,  for $i =0, 2$.  We then have 
\begin{align}
\label{EElow}
\| E_\low \|_{2, \alpha; \partial D} \leq \| E\|_{C^0(\partial D)} \leq C\|E\|_{1, \alpha; \partial D}, 
\end{align}
and by Lemma \ref{Lext0}(iv) and \eqref{EElow} for estimating $E_{\high} = E- E_\low$, that 
\begin{align}
\label{EElow1}
\|(\partialnu \Hdelta-1)^{-1} E_\high \|_{2, \alpha; \partial D} \leq C \| E_\high \|_{1, \alpha; \partial D} \leq C\|E\|_{1, \alpha; \partial D}. 
\end{align}
Combining \eqref{EElow} and \eqref{EElow1} with the definitions proves $\Rcalappr$ is bounded.

\emph{Step 2: Estimates on $\Bcal \Rcalappr - 1$}. 
Because $v = v_\low + v_\high$, where 
\[
v_\low = - E_\low, 
\quad 
v_\high = (\partialnu \Hdelta -1)^{-1} E_\high,
\]
it follows that $\partialnu \Hdelta v_\high - v = E$.  Consequently, 
\begin{align*}
\partialnu H_\Lcal v - v=\partialnu H_\Lcal v  -\partialnu \Hdelta v_{\high} + E,
\end{align*}
and taking oscillatory parts and rearranging shows that
\begin{align}
\label{EE1ERR}
\Bcal v - E =  (\partialnu H_\Lcal v - \partialnu \Hdelta v_{\high} )_\osc.
\end{align}
From \eqref{EE1ERR}, Proposition \ref{Pext1}(iii), and the boundedness of $\Rcalappr$, we conclude
\begin{equation}
\label{EBv1}
\begin{aligned}
\| \Bcal v- E\|_{1, \alpha; \partial D} \leq  \frac{C}{\sqrt{m}} \| v \|_{2, \alpha; \partial D} \leq  \frac{C}{\sqrt{m}} \| E\|_{1, \alpha; \partial D}.
\end{aligned}
\end{equation}

\emph{Step 3: The exact solution operator.}
By \eqref{EBv1}, the operator $\Bcal \Rcalappr - 1$ on $C^{1, \alpha}_{\sym, \osc}(\partial D)$ has norm bounded by $C/ \sqrt{m}$;  hence, $\Bcal \Rcalappr$ is a perturbation of the identity, with uniformly bounded inverse.  The proof is completed by defining $\Rcal = \Rcalappr (\Bcal \Rcalappr)^{-1}$ and using the preceding facts. 
\end{proof}

\section{Perturbations of $\Sph^2 \setminus D$}
\label{Spert}

\subsection{Definitions}
\begin{definition}
\label{dVv}
Given a real-valued function $v$ defined on $\partial D$, let $V_v$ be the vector field supported on $A \subset \Sph^2$ defined by
requesting that 
\begin{align}
\label{ParamDef}
V_v = \Psibold \left[ 0, 1/2;   \dbold_{\partial D}/\tau \right] \left(\pi^*v \,(- \nabla \dbold_L), 0\right) 
\end{align}
on $A$ (recall \ref{dAD}), where $\pi$ is the nearest-point projection to $\partial D$. 
\end{definition}
\begin{notation}
\label{Notv}
If $v$ and $V_v$ are as in Definition \ref{dVv} and $f$ is a function defined on a subset of $\Sph^2$ containing $\partial D$, let
\begin{equation*}
\begin{gathered}
g_v : = (\exp V_v)^* g, 
\quad
D_v : = (\exp V_v )( D), 
\quad
f_v : = (\exp V_v)^* f ,
\\
f^\partial_v : = f_v |_{\partial D}, 
\quad
f^\partial_{v, \osc} : = (f^\partial_v)_{\osc}, 
\quad
f^{\partial}_{v, \ave}: = (f^\partial_v)_{\ave},
\end{gathered}
\end{equation*}
where $\exp : \Sph^2 \rightarrow \Sph^2$ is the exponential map with respect to the metric $g$. 
Furthermore, whenever $g_v$ is a Riemannian metric, we use the notation
\begin{align*}
\Lcal_v : = \Delta_{g_v}+2, 
\quad
\Lcalhat_v : = \tau^2 \Lcal_v = \Delta_{\ghat_v} + 2\tau^2,
\end{align*}
denote by $\nu_v$ the $g_v$-unit outward pointing normal to $\Sph^2 \setminus D$, and set $\hat{\nu}_v = \tau \nu_v$.
\end{notation}

It is easy to see that $\exp V_v : \Sph^2 \rightarrow \Sph^2$ is a diffeomorphism whenever $v \in C^{2, \alpha}_{\sym}(\partial D)$ satisfies $\| v \|_{2, \alpha; \partial D}< \tau^2_2$.  In the remainder of the section we assume such a function $v$ is given, although we will frequently remind the reader about the smallness condition on the norm.  

\begin{remark}
\label{Requiv}
As in other papers \cite{PacardSic} constructing solutions of overdetermined problems by perturbative methods, we study eigenfunctions on perturbed domains $\Sph^2 \setminus D_v$ through equivalent problems on the fixed domain $\Sph^2 \setminus D$.  Specifically, since $\exp V_v$ induces an isometry from $(\Sph^2 \setminus D, g)$ to $(\Sph^2 \setminus D_v, g_v)$, the problems 
\begin{equation*}
\begin{cases}
(\Delta_g +2) \phi = 0 \hfill \quad &\text{in} \quad \Sph^2 \setminus D_{v}\\
\phi= 0  \hfill \quad &\text{on} \quad \partial D_{v}\\
| d \phi|_g = c \hfill \quad &\text{on} \quad \partial D_{v}.
\end{cases}
\end{equation*}
and
\begin{equation*}
\begin{cases}
(\Delta_{g_v} +2) \phi_v = 0 \hfill \quad &\text{in} \quad \Sph^2 \setminus D\\
\phi_v= 0  \hfill \quad &\text{on} \quad \partial D\\
| d \phi_v|_{g_v} = c \hfill \quad &\text{on} \quad \partial D, 
\end{cases}
\end{equation*}
are equivalent, where $\phi_v: = (\exp V_v)^* \phi$.
\end{remark}

\subsection{Basic estimates}
\begin{lemma}
\label{Lphibdest}
$\varphi_v$ is a $C^1$ function of $v$ and satisfies the following.
\begin{enumerate}[label=\emph{(\roman*)}]
\item $\| \varphibd_0\|_{2, \alpha; \partial D} \leq C\tau^{5/2}_2$. 
\item $\| (\varphibd_v)'(v)h + h \|_{2, \alpha; \partial D} \leq C \tau_2 \| h \|_{2,\alpha; \partial D}$. 
\item $ \|  (\varphi^\partial_{v, \osc})'(v)h \|_{2, \alpha; \partial D} \leq \| h_\osc \|_{2, \alpha; \partial D} + C \tau_2 \| h\|_{2, \alpha; \partial D}$.
\item $\| \varphibd_{v, \osc}\|_{2, \alpha; \partial D} \leq \| v_\osc \|_{2, \alpha; \partial D} + C \tau_2 \| v\|_{2, \alpha; \partial D}+ C\tau^{5/2}_2$. 
\end{enumerate}
\end{lemma}
\begin{proof}
Recalling Lemma  \ref{Lphierr} and Notation \ref{Notv}, we have
\begin{align}
\label{EphibdT}
\varphibd_v = \tau \log ( 1- \vhat) + \phierr_v. 
\end{align}
From this, we have $\varphi^{\partial}_0 = \phierr|_{\partial D}$, and (i) then follows from Lemma \ref{Lphierr}.

Next, because $\phierr_v = (\exp V_v)^* \phierr$, a calculation using \eqref{EphibdT} shows that
	\begin{align*}
	(\varphibd_v)'(v) h =  - h(1-\vhat)^{-1} -  \hat{h} \iota^* (\exp V_v)^* \partial_{\rhat} \phierr,
	\end{align*}
and after estimating that
\begin{align*}
\| (\varphibd_v)'(v) h + h \|_{2, \alpha; \partial D} \leq C \| h\|_{2, \alpha; \partial D} ( \| \vhat \|_{2, \alpha; \partial D} +  \| \hat{\varphi}^{err}\|_{3, \alpha; A}).
\end{align*}	
Using the bound $\| v \|_{2, \alpha; \partial D} < \tau^2_2$ and the estimates on $\phierr$ in Lemma \ref{Lphierr}, we conclude (ii).  Item (iii) follows from (ii), and (iv) follows from (i), (iii), and the mean value inequality. 
\end{proof}

\begin{lemma}
\label{Llapdiff}
If $\| v \|_{2, \alpha; \partial D} < \tau^2_2$, then the following hold. 
\begin{enumerate}[label=\emph{(\roman*)}]
\item $\| \ghat(\hat{\nu}_v, \hat{\nu}) - 1 \|_{1, \alpha; \partial D} \leq C \| d\vhat \|^2_{1, \alpha; \partial D} \leq C \tau^2_2$.
\item $\| \Lcalhat_{v}'(v) w \phi \|_{0, \alpha; A} \leq C \| \what \|_{2, \alpha; \partial D} \| \phi \|_{2, \alpha; A}$
\item If also $\| w\|_{2, \alpha; \partial D}  < \tau^2_2$, then for each $u \in C^{2, \alpha}_{\sym}(\Sph^2 \setminus D)$,
	\begin{align*}
	\| (\Lcalhat_v - \Lcalhat_w) u \|_{0, \alpha; \Sph^2 \setminus D} \leq C \| \vhat-\what \|_{2, \alpha; \partial D} \| u \|_{2,\alpha; A}.
	\end{align*}
\end{enumerate}
\end{lemma}
\begin{proof}
A general formula for the upward-pointing unit normal associated to an exponential normal perturbation of a $2$-sided hypersurface in a Riemannian manifold can be found in \cite[Corollary B.9]{LDG} or in \cite[Lemma 4.42]{Wiygul}.  Applying this to the case at hand, we conclude that
\begin{align*}
\nu_v = \frac{\nu - \nabla^{g^v} v}{(1+ |d v|^2_{g^v})^{1/2}},
\end{align*}
where $g^v$ is the metric on $\partial D$ defined by $g^v:= \frac{\sin^2(\tau-v)}{\sin^2(\tau)} g$.  From this and the definitions, item (i) follows. 

For (ii), parametrize a neighborhood of $\partial D$ in $\Sph^2$ over   $\partial D \times (-\epsilon, \epsilon)$ by the map $(p,z) \mapsto \exp_p( z \nu(p))$; locally, the metrics $g, g_v, \ghat_v$ then satisfy
\begin{equation}
\label{Egv}
\begin{aligned}
g &= dz^2 + \frac{\sin^2(\tau - z)}{\sin^2 \tau} g_{\partial D}, 
\\
g_v &= d(z + \psi v)^2 + \frac{\sin^2( \tau - z - \psi v)}{\sin^2 \tau} g_{\partial D}, 
\\
\ghat_v &= d( \zhat + \psi \vhat )^2 + \frac{\sin^2 (\tau ( 1 -\zhat  -\psi \vhat))}{\sin^2 \tau} \ghat_{\partial D}, 
\end{aligned}
\end{equation}
where $\psi \in C^\infty(\R)$ is supported on $(-\tau/2, \tau/2)$ and is identically $1$ on a neighborhood of $0$.  
Consequently, a short calculation reveals that
\begin{align*}
\ghat'_v (v)w &= 2 d\zhat d(\psi \what) + \frac{ \tau \sin (2\tau ( 1- \zhat - \psi \vhat))}{\sin^2 \tau} \ghat_{\partial D},
\end{align*}
and using the uniform estimates on $\psi$ with respect to $\ghat$, we then estimate $\| \ghat_v'(v)w\|_{2,\alpha; A} \leq C \| \what \|_{2, \alpha; \partial D}$.

Next,  in general, if $g$ is a metric on a manifold, the derivative $\Delta'_{g}$ with respect to the metric is given by \cite{Berger}
\begin{align}
\label{Egderiv}
\Delta'(g) h \phi = \langle D^2 \phi, h\rangle_g - \langle d \phi,  \mathrm{div}_g h + \frac{1}{2} d ( \mathrm{tr}_g h) \rangle_{g}. 
\end{align}
Applying this estimate in the case at hand, using the chain rule, and estimating using the bound for $\ghat'_v(v)w$ establishes
\begin{align*}
\| \Delta_{\ghat_v}'(v) w \phi \|_{0, \alpha; A} \leq C \| \what \|_{2, \alpha; \partial D} \| \phi \|_{2, \alpha; A}.
\end{align*}
Item (ii) follows easily from this, using that $\Lcalhat_{v} = \Delta_{\ghat_v} + 2\tau^2$.  Finally, item (iii) follows from (ii) using the mean value inequality.
\end{proof}

We next construct a modification $J_{\Lcal_v}$ of the operator $J_\Lcal$ from Proposition \ref{Pinhomog} which is adapted to the operator $\Lcal_v$. 
\begin{corollary}
\label{CJ}
Assuming $\| v \|_{2, \alpha; \partial D} < \tau^2_2$, there is a bounded linear map
\begin{align}
\label{EJLcalv}
\begin{gathered}
J_{\Lcal_v} :  \left\{ E \in C^{0, \alpha}_{\sym}(\Sph^2 \setminus D) : \mathrm{supp} \, E \subset A \setminus D)\right\} 
\rightarrow C^{2, \alpha}_{\sym}(\Sph^2 \setminus D), 
\\
J_{\Lcal_v} := J_{\Lcal} [ 1-  (\Lcalhat - \Lcalhat_v)J_\Lcal]^{-1}
\end{gathered}
\end{align}
such that if $E \in \mathrm{dom}\,  J_{\Lcal_v}$ and $u= J_{\Lcal_v} E$, then the following hold. 
\begin{enumerate}[label=\emph{(\roman*)}]
\item $\Lcalhat_v u = E$ on $\Sph^2 \setminus D$. 
\item $(u|_{\partial D})_\osc = 0$ and  $|(u|_{\partial D})_{\ave}|   \leq C \| E \|_{0, \alpha; A \setminus D}$.
\end{enumerate}
\end{corollary}
\begin{proof}
Because $\| v \|_{2, \alpha; \partial D} < \tau^2_2$,  Lemma \ref{Llapdiff}(iii) and Proposition \ref{Pinhomog} imply
$(\Lcalhat - \Lcalhat_v)J_{\Lcal}$ has operator norm bounded by $C \| \vhat \|_{2, \alpha; \partial D} << 1$.  Therefore $J_{\Lcal_v}$
is well-defined and bounded, and the conclusion follows from the definition \eqref{EJLcalv} and the properties of $J_\Lcal$ in Proposition \ref{Pinhomog}. 
\end{proof}

\subsection{The space of nearby eigenfunctions} We first define some auxiliary functions which will be used to assemble the first eigenfunction on $\Sph^2 \setminus D_v$.

\begin{lemma}
\label{Lphiv}
\label{dphi}
Whenever $\| v \|_{2, \alpha; \partial D} < \tau^2_2$, the function 
\begin{equation}
\label{Ephiv0}
\begin{aligned}
\phi_v : = \varphi_v - \xiv - \xivtilde \in C^{2, \alpha}_{\sym}(\Sph^2 \setminus D),
\end{aligned}
\end{equation}
where $\xiv, \xivtilde \in C^{2, \alpha}_{\sym}(\Sph^2 \setminus D)$ are defined by
\begin{align*}
\xiv: = H_\Lcal \varphibd_{v, \osc},
\quad
\xivtilde : = J_{\Lcal_v} ( \Lcalhat - \Lcalhat_v) \xiv,
\end{align*}
is constant on $\partial D$ and satisfies $\Lcal_v \phi_v =0$ on $\Sph^2 \setminus D$. 
\end{lemma}
\begin{proof}
From $\Lcal \varphi = 0$, Proposition \ref{Pext1}(i), and Corollary \ref{CJ}(i), we have
\begin{align*}
\Lcal_v \varphi_v = 0, 
\quad
\Lcal \xiv  = 0, 
\quad
\Lcalhat_v \xivtilde = (\Lcalhat - \Lcalhat_v) \xiv = - \Lcalhat_v \xiv,
\end{align*}
and hence $\Lcal_v \phi_v = 0$. 
Next, by Proposition \ref{Pext1} and Corollary \ref{CJ}(iii),
\begin{align*}
\xi^\partial_{v, \osc} = \varphibd_{v, \osc} 
\quad
\text{and}
\quad
\xitilde^\partial_{v, \osc} = 0
\end{align*}
so that $\phi^\partial_v = \varphibd_{v,\ave} - \xi^\partial_{v, \ave} - \xitilde^\partial_{v, \ave}$, proving $\phi_v$ is constant on $\partial D$. 
\end{proof}

Using the estimates on $\varphibd_v$ from Lemma \ref{Lphibdest}, we now estimate $\xiv^\partial$ and $\xivtilde^\partial$. 
\begin{lemma}[Estimates at the boundary]
\label{Lphibd}
If $\| v \|_{2, \alpha; \partial D} < \tau^2_2$, then 
\begin{enumerate}[label=\emph{(\roman*)}]
\item $|(\xi^\partial_{v,\ave})'(v)h | \leq  \frac{C}{\sqrt{m}}( \| h_\osc \|_{2, \alpha; \partial D} + \tau_2 \| h \|_{2, \alpha; \partial D})$.
\item $| (\xitilde^{\partial}_{v,\ave})'(v) h |  \leq C \tau_2 \| h \|_{2, \alpha; \partial D}$. 
\item $| (\phi_v^\partial)'(v)h + h_\ave| \leq C ( \tau_2 \| h \|_{2, \alpha; \partial D} + \frac{1}{\sqrt{m}} \| h_\osc\|_{2, \alpha; \partial D})$. 
\end{enumerate}
\end{lemma}
\begin{proof}

Because $\xiv = H_\Lcal \varphibd_{v, \osc}$, the linearity of $H_\Lcal$ and the estimate on the average part in Proposition \ref{Pext1}(ii) imply
\begin{align*}
|( \xi^{\partial}_{v, \ave})'(v) h | \leq  \frac{C}{\sqrt{m}} \| (\varphi^{\partial}_{v, \osc})'(v) h \|_{2, \alpha; \partial D},
\end{align*}
and item (i) follows from this by Lemma \ref{Lphibdest}(iii).

Because $\xivtilde = J_{\Lcal_v} (\Lcalhat - \Lcalhat_v) \xiv$,  the product rule implies
\begin{align*}
\begin{gathered}
(\xivtilde)'(v)h = I + II + III, \quad \text{where} 
\quad
I = (J_{\Lcal_v})'(v) h (\Lcalhat - \Lcalhat_v) \xiv ,
\\
II = - J_{\Lcal_v}(\Lcalhat_v)'(v)h \, \xiv, 
\quad
III = J_{\Lcal_v}(\Lcalhat - \Lcalhat_v ) (\xiv)'(v)h, 
\end{gathered}
\end{align*}
and (ii) follows by estimating  using Lemma \ref{Lphibdest}(iii), Lemma \ref{Llapdiff}(ii), and \eqref{EJLcalv}.

Finally, because $\phi_v= \varphi_v - \xiv -\xivtilde$, item (iii) follows from (i)-(ii) and Lemma \ref{Lphibdest}.
\end{proof}

\begin{prop}
\label{Pphiv}
There is a $C^1$ map
\begin{align*}
f: \{ w \in C^{2, \alpha}_{\sym, \osc}(\partial D) : \| w \|_{2, \alpha; \partial D} < \tau^2_2/2 \}   \rightarrow \R
\end{align*}
 uniquely determined by the property that $\phi_{w + f(w)}$ vanishes on $\partial D$.  Moreover,
 \begin{enumerate}[label=\emph{(\roman*)}]
\item $|f(0) | \leq C \tau^{5/2}_2$. 
\item $|f(w)-f(0)| \leq  \frac{C}{\sqrt{m}}\|w\|_{2, \alpha;\partial D}$ and $\| w + f(w) - f(0) \|_{2, \alpha; \partial D}\leq \tau^2_2$.
\end{enumerate}
\end{prop}
\begin{proof}
Consider the map $F$ defined by $F(w, c) = \phi^\partial_{w+c}$ on the open set
\[ 
\{ w \in C^{2, \alpha}_{\sym, \osc}(\partial D) : \| w \|_{2, \alpha; \partial D} < \tau^2_2/2 \} \times \{ c \in \R: |c| < \tau^{5/2}_2\},
\] 
and note from Lemma \ref{Lphiv} that $F$ maps into $\R$.  The partial derivatives $(D_1 F)(w,c)$ and $(D_2 F)(w, c)$ satisfy
\begin{align*}
(D_1 F)(w, c) u = \left. \frac{d}{dt} \right|_{t = 0} \phi^\partial_{w+ tu + c} , 
\quad
(D_2 F)(w, c) = \left. \frac{d}{dt}\right|_{t =0} \phi^\partial_{w+c+t}, 
\end{align*}
and estimating using Lemma \ref{Lphibd}(iii) shows that
\begin{equation}
\label{Epartials}
\begin{gathered}
|(D_1 F)(w, c) u | \leq  \frac{C}{\sqrt{m}} \| u \|_{2, \alpha; \partial D}, 
\quad
(D_2 F)(w, c) = -1 + O(\tau_2).
\end{gathered}
\end{equation}

On the other hand, by Lemma \ref{Lphibdest}(i) to estimate $\varphi^\partial_0$,  an obvious estimate on $\xi^\partial_0$ using Lemma \ref{Pext1}, and the fact that $\xitilde_0 = 0$, we have $|\phi^\partial_0|< C \tau^{5/2}_2$.  It follows there is a number $\cunder$ with $|\cunder| < C \tau^{5/2}_2$ such that $F(0, \cunder) = 0$, or equivalently, that $\phi^\partial_{\cunder} = 0$.  The existence and uniqueness of $f$ now follows from the implicit function theorem, and (i) follows from the preceding.

By the Implicit function theorem, the derivative $f'(w)$ satisfies
\begin{align}
f'(w) =  - [ (D_2 F)(w, f(w))]^{-1} (D_1 F)(w, f(w)), 
\end{align}
and estimating this using \eqref{Epartials} shows that
\begin{align*}
| f'(w) u | \leq  \frac{C}{\sqrt{m}} \| u \|_{2, \alpha; \partial D}. 
\end{align*}
Combined with the mean value inequality, this implies (ii).
\end{proof}

\begin{prop}
\label{LinN}
The map 
\begin{align*}
&N : \{ w \in C^{2, \alpha}_{\sym, \osc}(\partial D) : \| w \|_{2, \alpha; \partial D} < \tau^2_2/2 \}  \rightarrow C^{1, \alpha}_{\sym, \osc}(\partial D),
\\
&N( w) = ( \nuhat_v \phi_v)_{\osc}, \quad \text{where} \quad v: = w+f(w)
\end{align*}
and $f$ is as in Proposition \ref{Pphiv}, satisfies the following. 
\begin{enumerate}[label=\emph{(\roman*)}]
\item $\| N(0) \|_{1, \alpha; \partial D } \leq C \tau^{5/2}_2$. 
\item $\| N(w) - N(0) -  \Bcal w\|_{1, \alpha; \partial D} \leq C \tau^3_2$, where $\Bcal$ is as in \eqref{dBcal}.
\end{enumerate}
\end{prop}
\begin{proof}
For $w, v, \phi_v$ as above, expanding $\phi_{v}$ using \eqref{Ephierr} and \eqref{Ephiv0} reveals that
\begin{equation}
\label{Ephihat}
\begin{aligned}
\phi_v &= \tau\log(1-\zhat - \vhat) + \phierr_{v} - H_\Lcal \varphibd_{v, \osc} - \xivtilde,
\\
\partialnu \phi_v &= -\tau (1-\vhat)^{-1} + \partialnu \phierr_{v} - \partialnu H_\Lcal \varphibd_{v, \osc} - \partialnu  \xivtilde,
\end{aligned}
\end{equation}
where $z$ is the signed distance from $\partial D$ as in Lemma \ref{Llapdiff}.  

In particular, when $v = v(0) : = c = f(0)$ (recall Proposition  \ref{Pphiv}), then Lemma \ref{Llapdiff} and the fact that $v = \cunder$ is constant implies $\nuhat_{\cunder} = \nuhat = \partial_{\zhat}|_{\partial D}$, so  
\begin{align}
\label{EN0}
N(0) =( \nuhat \phi_{\cunder})_\osc =   (\partialnu \phierr_{\cunder} - \partialnu H_\Lcal \varphibd_{\cunder, \osc} - \partialnu \xitilde_{\cunder})_{\osc}.
\end{align}
Estimating using Proposition \ref{Pext1}, Lemma \ref{Llapdiff}, and Corollary \ref{CJ} shows that
\begin{align*}
\| N(0) \|_{1, \alpha; \partial D} 
\leq  C\| \phierr \|_{3, \alpha; A} + C\| \varphibd_{\cunder, \osc}\|_{2, \alpha; \partial D},
\end{align*}
 and  (i) follows from this using Lemma  \ref{Lphierr} and Lemma \ref{Lphibdest}(iv). 

For (ii), since $\phi_v$ vanishes along $\partial D$, it follows that
\begin{align*}
(\hat{\nu}_v - \hat{\nu}) \phi_v  = ( \langle \hat{\nu}_v, \hat{\nu}\rangle_{\ghat} - 1 )  \partialnu \phi_v ,
\end{align*}
and hence from this and Lemma \ref{Llapdiff}(i) that 
\begin{equation}
\begin{aligned}
\label{ENv1}
\| N(w) - (\partialnu \phi_v)_{\osc} \|_{1, \alpha; \partial D} 
\leq C \tau^2_2  \| \partialnu \phi_v \|_{1, \alpha; \partial D}
\leq C \tau^3_2,
\end{aligned}
\end{equation}
where the second inequality estimates $\partialnu \phi_v$ using \eqref{Ephihat}.

Using \eqref{dBcal}, \eqref{Ephihat}, and \eqref{EN0}, we  find 
\begin{multline*}
(\partialnu \phi_v)_{\osc} - \Bcal w - N(0) = (-\tau(1-\vhat)^{-1} + \partialnu \phierr_{v} - \partialnu H_\Lcal \varphibd_{v, \osc} - \partialnu  \xivtilde
\\
-\hat{\nu} H_\Lcal w + w -\partialnu \phierr_{\cunder} + \partialnu H_\Lcal \varphibd_{\cunder, \osc} + \partialnu \xitilde_{\cunder} )_\osc
\end{multline*}
and using that $w = v - c = v_\osc$ and rearranging, we see
\begin{equation*}
\begin{gathered}
(\partialnu \phi_v)_{\osc} - \Bcal w - N(0) = (I + II - III - IV)_\osc,
\quad
\text{where}
\\
I: =  v - \tau(1-\vhat)^{-1},
\quad
II: =  \partialnu (\phierr_{v} -  \phierr_{\cunder}),
\\
III: =  \partialnu H_{\Lcal}(\varphibd_v + v - \varphibd_{\cunder} - \cunder)_\osc,
\quad
IV: = \partialnu ( \xivtilde - \xitilde_{\cunder}) .
\end{gathered}
\end{equation*}

First, we estimate
\begin{align*}
\| I_\osc \|_{1, \alpha; \partial D} & \leq C \| \tau \vhat^2 \|_{2, \alpha; \partial D} 
\leq C \tau^3_2,
\end{align*}
where we have used that $\|v\|_{2, \alpha; \partial D} \leq C\| w\|_{2, \alpha; \partial D}$ from Proposition \ref{Pphiv}.

We estimate the remaining error terms as follows.  Using Lemma \ref{Lphierr} and Taylor's theorem for $II$;  the boundedness of $H_\Lcal$, \ref{Lphibdest}(ii), and the mean value inequality for $III$; and Lemma \ref{Lphibd}(ii) and the mean value inequality for $IV$, we estimate
\begin{equation*}
\begin{aligned}
\| II \|_{1, \alpha; \partial D} &\leq \|\phierr_{v} -  \phierr_{\cunder}\|_{2, \alpha; \partial D} \leq C \tau_2^{5/2} \| v - \cunder\|_{2, \alpha; \partial D},
\\
\| III \|_{1, \alpha; \partial D} &\leq C \| (\varphibd_v + v - \varphibd_{\cunder} - \cunder)_\osc \|_{2, \alpha; \partial D} \leq C \tau_2 \| v - \cunder \|_{2, \alpha; \partial D},
\\
\| IV \|_{1, \alpha; \partial D} 
&\leq 
C \| \xivtilde - \xitilde_{c}  \|_{2, \alpha; \Sph^2 \setminus D}
\leq C \tau_2 \| v - \cunder\|_{2, \alpha; \partial D}
\end{aligned}
\end{equation*}
and recall from Proposition \ref{Pphiv}(ii) that $\| v - \cunder \|_{2, \alpha; \partial D} \leq C\tau^2_2$.

By combining the preceding, we have
\begin{align}
\label{ENv2}
\| (\partialnu \phi_v)_{\osc} - N(0) - \Bcal w \|_{1, \alpha; \partial D} \leq C \tau^3_2. 
\end{align}
Item (ii) now follows by combining \eqref{ENv1} and \eqref{ENv2}.
\end{proof}

\subsection{Main results}
\begin{theorem}
\label{Tmain}
There is a number $m_0$ such that if $m> m_0$, 
\begin{align*}
L = L[m] = L_0[m] \cup L_2 \subset \Sph^2
\end{align*}
is the $\group$-invariant set of $m+2$ points as in \ref{dL}, the numbers
\begin{align*}
\tau_0 &=e^{ - \sqrt{m/2} - \frac{3}{4} \log m + \zeta - \Phip_0(p_0)}, 
\\
\tau_2 &= \tau_0 ( \sqrt{m/2} - \frac{1}{4} \log m - \zeta )
\end{align*}
are as in Lemma \ref{Lphitau}, and the neighborhood $D = D_{L_0}(\tau_0) \cup D_{L_2}( \tau_2)$ of $L$ is as in \ref{dAD}, then 
there is a $\group$-invariant function $v \in C^{2, \alpha}(\partial D)$ satisfying
\begin{align*}
\| v \|_{2, \alpha; \partial D}
\leq 
C \tau^{5/2}_2
\end{align*}
such that the perturbation $\Omega$ of $\Sph^2 \setminus D$ with boundary the normal graph
\begin{align*}
\partial \Omega : = \{ \exp_p( v(p) \nu(p) ) : p \in \partial D\}
\end{align*}
is a $\lambda_1$-extremal domain in $\Sph^2$ with $\lambda_1(\Omega) = 2$.  In particular, $\Omega$ admits a solution to the overdetermined problem \eqref{EoverOG}, is $\group$-symmetric, and has real-analytic boundary, consisting of $m+2$ components.
\end{theorem}
\begin{proof}
By \ref{Plin}, the function $w_0 : =  - \Rcal N(0) \in C^{2, \alpha}_{\sym, \osc}(\partial D)$ satisfies
\begin{align}
\label{Ev0eq}
\Bcal w_0 = - N(0), 
\end{align}
and, by virtue of Proposition \ref{Plin} and Proposition \ref{LinN}(i), the estimate
\begin{align}
\label{Ev0est}
\| w_0\|_{2, \alpha; \partial D} \leq C \| N(0) \|_{1, \alpha; \partial D} \leq  C\tau^{5/2}_2.
\end{align}
Next, define a subset $B \subset C^{2, \alpha}_{\sym, \osc}(\partial D)$ by
\begin{align}
\label{EdefB}
B = \{ w \in C^{2, \alpha}_{\sym, \osc}(\partial D): \| w \|_{2, \alpha; \partial D} \leq \tau^{5/2}_{2}\}, 
\end{align}
and a map $\Jcal : B \rightarrow C^{2, \alpha}_{\sym, \osc}(\partial D)$ by
\begin{align}
\label{EJcal}
\Jcal(w) = - \Rcal \big( N(w_0 + w) - N(0) - \Bcal (w_0 +w)\big). 
\end{align}

Now, whenever $w \in B$,  \eqref{Ev0est} and \eqref{EdefB} imply
\[
\|w_0 + w \|_{2, \alpha; \partial D} \leq C \tau^{5/2}_2 < \tau^2_2/2;
\]
 applying Lemma \ref{LinN}(ii) and using \eqref{EJcal} and Proposition \ref{Plin} shows  that
\begin{equation*}
\begin{aligned}
\| \Jcal (w) \|_{2, \alpha; \partial D} & \leq C\| N(w_0+w) - N(0) - \Bcal(w_0 + w) \|_{1, \alpha; \partial D}
\\
&\leq C \tau_2^{3}.
\end{aligned}
\end{equation*}
It follows that $\Jcal(B) \subset B$.

Let $\beta \in (0,\alpha)$.  By the Arzela-Ascoli theorem, $B$ is a compact, and clearly convex, subset of $C^{2, \beta}_{\sym, \osc} (\partial D)$.  Furthermore, it follows from \eqref{EJcal} and the definitions of $N$ and $\Bcal$ in \ref{LinN} and \eqref{dBcal} that $\Jcal$ is continuous in the induced topology.

  The Schauder fixed-point theorem \cite[Theorem 11.1]{Gilbarg} now implies there is a fixed-point $w$ for $\Jcal$, which in view of \eqref{EJcal} and Proposition \ref{Plin}, satisfies 
\begin{align*}
\Bcal w = - \big(N(w_0 + w) - N(0) - \Bcal (w_0 +w)\big).
\end{align*}
With \eqref{Ev0eq}, this shows that $N(w_0+w) = 0$; by the definition of $N$ in \ref{LinN}, this means $\nuhat_v \phi_v$ is constant, where $v = w_0 + w + f(w_0 +w)$.  Because $\phi_v$ is zero on $\partial D$ (Proposition \ref{Pphiv}), this implies $| d \phi_v|_{g_v} = 0$ along $\partial D$. Additionally, by Lemma \ref{Lphiv}, $\phi_v$ satisfies $(\Delta_{g_{v}}+2) \phi_v = 0$ on $\Sph^2 \setminus D$.

 As a consequence of the preceding facts and Remark \ref{Requiv}, the function $\phi : = (\exp V_v)^{-1 *} \phi_v$ satisfies
\begin{equation*}
\begin{cases}
(\Delta_g +2) \phi = 0 \hfill \quad &\text{in} \quad \Sph^2 \setminus D_{v}\\
\phi= 0  \hfill \quad &\text{on} \quad \partial D_{v}\\
| d \phi|_g = c \hfill \quad &\text{on} \quad \partial D_{v}.
\end{cases}
\end{equation*}
Furthermore, by construction $\phi_v$ is nonnegative on $\Sph^2 \setminus D$, so $\phi$ solves the overdetermined system \eqref{EoverOG}.  

From the preceding, the boundary $\partial D_{v}$ of the domain $\Sph^2 \setminus D_{v}$ is of class $C^{2, \alpha}$; because $\phi$ solves \eqref{EoverOG}, it follows from this and standard regularity results \cite{Nirenberg} that $\partial D_{v}$ is analytic.
\end{proof}

\section{Results in dimension four}
\label{Sdim4}

Here we prove Theorem \ref{Tmain4}.  To describe the symmetries and to estimate the LD Solutions, we keep closely to parts of \cite{KapZou} where these details are discussed.  For clarity, however, we keep the exposition mostly self-contained.

We then adapt the approach taken in Sections \ref{Slin} and \ref{Spert} to prove Theorem \ref{Tmain}.  Because of the high symmetry, many of the arguments are analogous to ones from earlier, and to avoid unnecessary repetition, we omit proofs for arguments which involve only obvious notational changes.

\subsection{Geometry of $\Sph^3$}
\label{ssS3}
Denote by $\Sph^3$ the unit sphere in $\R^4$, and by $g$ the round metric on $\Sph^3$ induced by the Euclidean metric on $\R^4$.  To simplify notation, we identify $\R^4$ with $\C^2$ and denote by $(z_1, z_2)$ the standard coordinates on $\C^2$. 
The Clifford torus $\T \subset \Sph^3$ is defined by
\begin{align*}
\T = \{ (z_1, z_2) \in \C^2 : |z_1| = |z_2| = 1/ \sqrt{2}\},
\end{align*}
and we will also refer to orthogonal great circles $\uC$ and $\uC^\perp$ in $\Sph^3$ defined by
\begin{equation} 
\label{Ecliff}
	  \begin{aligned}
	  \uC:=\Sph^3\cap\{z_1=0\},
	  \quad
	  \uC^{\perp}:=\Sph^3 \cap\{z_2=0\}.
	 \end{aligned} 
	 \end{equation}

\begin{notation}\label{Ngrp}
From the identification $\R^4\cong\R^2\times \R^2$, we consider the embedding $O(2)\times O(2)\hookrightarrow O(4)$ by the standard action of $O(2)$ on each component of $\R^2$ respectively, and denote the image of $O(2)\times O(2)$ in $O(4)$ by $\Hscr$.
\end{notation}

Just as before, $m$ will denote a positive integer which can be taken as large as needed in terms of absolute constants. 

\subsection{The symmetries and the configurations}

\begin{definition}
\label{dL3}
Define a set $L = L[m] \subset \Sph^3$ of $m^2$ points and $p_0 \in L$ by 
\begin{equation}
\label{EL3}
\begin{gathered}
L = \Big\{ \frac{1}{\sqrt{2}} ( e ^{i \frac{2\pi j}{m}} , e ^{i \frac{2\pi k}{m}} ): j, k \in \Z\Big \}, 
\\
p_0 : = (1/\sqrt{2},0,1/ \sqrt{2},0) \in L
\end{gathered}
\end{equation}
and let $\group, \groupT$ be the subgroups of $O(4)$ fixing the sets $L$ and $\T$, respectively.  
\end{definition}

\begin{lemma}[Properties of $\group$]
\label{Lgroup3}
The following hold.
\begin{enumerate}[label=\emph{(\roman*)}]
	\item $\mathscr{G}_{\mathbb{T}}$ is generated by $\Hscr$ and the involution $\SSS\in O(4)$ defined by 
	\[ \SSS(z_1, z_2) = (z_2, z_1).\]
	\item $\group$ is generated by $\SSS$ and reflections $\XXXu_0,\XXXu_{\pi/m},\YYYu_0,\YYYu_{\pi/m}$ defined by
		\begin{align*}
		\begin{gathered}
		\XXXu_0 (z_1, z_2) : =(z_1, \bar{z}_2 ), 
		\quad
		\XXXu_{\pi/m} (z_1, z_2) : = (z_1, e^{i \frac{2\pi}{m} }\bar{z}_2 ), 
		\\
		\YYYu_0(z_1, z_2) : = (\bar{z}_1, z_2), 
		\quad
		\YYYu_{\pi/m} (z_1, z_2) : = (e^{ i \frac{2\pi}{m}}\bar{z}_1, z_2),
		\end{gathered}
		\end{align*}
		and acts transitively on $L$.
	\item The derivative of any $\group$-symmetric differentiable function vanishes on $L$.
\end{enumerate} 	
\end{lemma}
\begin{proof}
This follows easily from the definitions, and we omit the details. 
\end{proof}

\begin{notation}[Symmetric functions] 
	 	\label{Ngroup4} 
	 	If $X$ is a function space consisting of functions defined on a $\group$-invariant domain $\Omega \subset \Sph^3$, we use a subscript ``sym" to denote the subspace $X_\sym \subset X$ of $\group$-invariant functions.
	 \end{notation}

Analogously to Definition \ref{dharm}, whenever $S \subset \Sph^3$ is a round sphere, let $\Hcal^k(S)$ denote the $k$-th nontrivial Laplacian eigenspace on $S$.  We also define
\begin{align*}
\Hcal^k_{\sym}(X) = \{ x \in C^\infty_{\sym}(X) : u|_S \in \Hcal^k(S) \quad \text{for each} \quad S \in X\} .
\end{align*}
whenever $X$ is a finite, $\group$-invariant set of pairwise disjoint round spheres. 

\begin{lemma}
\label{Lharm04}
If $r \in (0, 1/m )$, then
\begin{enumerate}[label=\emph{(\roman*)}]
	\item $\Hcal^0_{\sym}(\partial D_{L}(r))$ is $1$-dimensional.
	\item $\Hcal^1_{\sym}(\partial D_{L}(r))$ is $0$-dimensional.
	\end{enumerate}
\end{lemma}
\begin{proof}
This is clear from Lemma \ref{Lgroup3}.
\end{proof}

\subsection{The operator $\Lcal = \Delta + 3$}

\begin{definition}
Denote by $\Lcal$ the operator $\Delta + 3$ on $\Sph^3$, where $\Delta$ is the Laplace-Beltrami operator with respect to the usual metric $g$ on $\Sph^3$.
 \end{definition}
 
 Throughout, we will use the fact that $\ker \Lcal$ is spanned by the coordinate functions on $\Sph^3$, and in particular \cite[Lemma 2.2]{KapZou} that $(\mathrm{ker} \, \Lcal )_{\sym}$ is trivial.

 \begin{lemma}\label{Lgreen3}
	  	The function $G \in C^{\infty}((0,\pi))$ defined by
        \begin{align*}
G(r) = -\frac{\cos2r}{\sin r}
\end{align*}
has the following properties:
\begin{enumerate}[label=\emph{(\roman*)}]
	\item $\Lcal (G\circ \dbold_p)= 0$ on $\Sph^3 \setminus \{-p,p\}$, whenever $p \in \Sph^3$. 
         \item $G(r) = -(1 + O(r^2)) \frac{1}{r}$ for small $r> 0$.
	 \item $\|{G+1/r:C^k((0,1),r,dr^2,r)}\|\leq 1$.
\end{enumerate}
\end{lemma}
\begin{proof}
See \cite[Lemma 4.1]{KapZou}.
\end{proof}

As before, we define a scaled metric $\gtilde$ and scaled linear operator $\Lcaltilde$ by
\begin{align}
\label{Egtilde3}
\gtilde : = m^2 g, 
\quad
\Ltilde : = \Delta_{\gtilde} + 3 m^{-2} = m^{-2} \Lcal. 
\end{align}

\subsection{Rotationally invariant functions}
We call a $\group$-invariant function defined on a domain of $\Sph^3$ which only depends on the distance $\dbold_{\T}$ to the Clifford torus $\T$ a \textit{rotationally invariant function}. 

\begin{definition}\label{avgdef}
Given a $\group$-invariant function $\varphi$ on a domain $\Omega \subset \Sph^3$, we define a rotationally invariant function $\varphi_\ave$ on the union of $\Omega'$ of the parallel tori $\mathbb{T}_c$ on which $\varphi$ is integrable by requesting that 
\begin{align*}
\varphi_\ave|_{\mathbb{T}_c} := \ave_{\mathbb{T}_c} \varphi
\end{align*}
on each such torus.  We also define $\phi_\osc$ on $\Omega \cap \Omega'$ by $\varphi_\osc := \varphi - \varphi_\ave$.
\end{definition}

 If $\Omega$ is a $\mathscr{G}_\T$-invariant domain and $X_{\sym}$ is a space of $\group$-invariant functions defined on $\Omega$,  we use a subscript ``$\rot$" to denote the subspace $X_{\rot} \subset X$ of rotationally invariant functions, which therefore depend only on $\dbold_{\T}$.  
  
Note that a rotationally invariant solution to $\mathcal{L} \varphi = 0$ solves the ODE
\begin{equation}\label{Ejac}
\frac{d^2\varphi}{d\zz^2}-2\tan{2\zz}\frac{d\varphi}{d\zz}+3\varphi=0,
\end{equation}
where $\zz$ is a choice of signed distance to $\T$.

\begin{lemma} 
\label{LphiC}
	  	The space of solutions of the ODE \eqref{Ejac} in $\zz$ on $(-\pi/4,\pi/4)$ is spanned by functions $\phi_{\uC}, \phi_{\uC^\perp}$ with the following properties.
\begin{enumerate}[label=\emph{(\roman*)}]
\item  $\phi_\uC$ is singular at at $\{\zz=-\pi/4\}$ and is smooth at $\{\zz=\pi/4\}$, while \\
	 $\phi_{\uC^{\perp}}$ is singular at $\{\zz=\pi/4\}$ and is smooth at $\{\zz=-\pi/4\}$.
\item $\phi_\uC$ is strictly increasing in $\zz$ on $(-\pi/4,\pi/4)$, while \\ 
	$\phi_{\uC^{\perp}}$ is strictly decreasing in $\zz$ on $(-\pi/4,\pi/4)$.
\item $\phi_{\uC^{\perp}}(\zz)=\phi_\uC(-\zz)$.
\item $\phi_\uC(0)=\phi_{\uC^{\perp}}(0)=1$, $\phi'_\uC(0)=-\phi'_{\uC^{\perp}}(0)=F$ where $F \in (2.18, 2.19)$.
\end{enumerate}
	  \end{lemma}

\begin{proof}
This is essentially \cite[Lemma 2.7]{KapZou}.
\end{proof}

\subsection{LD Solutions}
\label{ssLD3}

\begin{lemma}
\label{Lphi3}
There is a unique function $\Phi = \Phi[m]$ such that 
	\begin{enumerate}[label=\emph{(\roman*)}]
	\item $\Phi \in C^\infty_{\sym}(\Sph^3 \setminus L)$, where $L= L[m]$ is as in \ref{dL3}; 
	\item $\Lcal \Phi = 0$ on $\Sph^3\setminus L$; and
	\item $\Phi+ 1/ \dbold_L$ is bounded on $\Sph^3$. 
	\end{enumerate}
\end{lemma}
\begin{proof}
This follows from Lemma 4.7 in \cite{KapZou}.
\end{proof}
\begin{lemma}[Characterization of $\Phi_\ave$]
\label{Lphiave3}
The following hold.
\begin{enumerate}[label=\emph{(\roman*)}]
\item $\Phi_\ave \in C^{0}(\Sph^3) \cap C^\infty(\Sph^3\setminus \T)$ and satisfies $\Lcal \Phi_\ave = 0$ on $\Sph^3 \setminus \T$. 
\item $\Phi_\ave = \frac{m^2}{\pi F} (\phi_{\uC} \circ \dbold_\T)$ and is a strictly increasing function of $\dbold_\T$.
\item The function $\phi_{\uC} \circ \dbold_\T$ in (ii) satisfies $\phi_{\uC} \circ \dbold_\T = \phiunder + \junder$, where
	\begin{align*}
	\phiunder &: = \frac{1}{2}( \phi_{\uC} + \phi_{\uC^\perp})\circ \dbold_\T \in C_{\rot}^\infty(\Sph^3\setminus (\uC \cup \uC^\perp)), 
	\\
	\junder &: =	\frac{1}{2}(\phi_\uC-\phi_{\uC^{\perp}}) \circ \dbold_\T \in C_{\rot}^\infty(\Sph^3 \setminus \T).
	\end{align*}
\item $\| \phiunder - 1 : C^k( D_{\T}(1/m), \dbold_\T, g, \dbold^2_\T) \| \leq C(k)$. 
\end{enumerate}	
\end{lemma}
\begin{proof}
Item (i) follows easily from Lemma \ref{Lphi3}(iii).  For (ii), it is clear from the symmetry that $\Phi_\ave = A (\phi_{\uC} \circ \dbold_\T)$ for some $A \in \R$.  For $0< \epsilon_1 << \epsilon_2$, integrating $\Lcal \Phi=0$ on the domain $D_{\T}(\epsilon_2) \setminus D_{L}(\epsilon_1)$, integrating by parts, and taking the limit as $\epsilon_1 \rightarrow 0$ first and then as $\epsilon_2 \rightarrow 0$, the $1/\dbold_{L}$ behavior near $L$ shows that
\begin{align*}
2 A F \mathrm{area}(\T) = m^2 \mathrm{area}(\Sph^2),
\end{align*}
and since $\mathrm{area}(\T) = 2\pi^2$ and $\mathrm{area}(\Sph^2) = 4\pi$, item (ii) follows.

Item (iii) follows from Lemma \ref{LphiC}.  Finally, since $\phiunder-1$ has vanishing value and differential along $\T$, item (iv) follows from the definitions and basic ODE theory. 
\end{proof}

In analogy to definition \ref{dGhat}, we define a decomposition $\Phi = \Ghat + \Phat + \Phip$, where $\Ghat$ contains $\Phi$'s singular part, $\Phat$ is smooth on $\Sph^3$ and rotationally invariant, and $\Phip$ is an error term.

\begin{convention}
In what follows, we let $\delta = 1/(10 m )$. 
\end{convention}

\begin{definition}
\label{ddecomp3}
Given $\Phi$ as in \ref{Lphi3}, we define
\begin{align*}
\Ghat \in C^\infty_{\sym}(\Sph^3 \setminus L), 
\quad
\Phat \in C^\infty_{\mathrm{rot}}(\Sph^3), 
\quad
\Phip \in C^\infty_{\sym}(\Sph^3), 
\quad
E' \in C^\infty_{\sym}(\Sph^3)
\end{align*}
by requesting that $\Ghat$ is supported on $D_L(3\delta) \setminus L$, where it is defined by
\begin{align}
\label{EGhat4}
	\Ghat:=\cutoff{2\delta}{3\delta}{\dbold_L}(G\circ \dbold_L,0)
	\quad
	\text{on}
	\quad
	D_L(3\delta),		
\end{align}
that $\Phat = \Phi_\ave$ on $\Sph^3 \setminus D_{\T} (1/m)$, that
\begin{align*}
\Phat = \Phi_\ave - \cutoff{\textstyle{\frac{1}{2m}}}{\frac{1}{m}}{\dbold_{\T}}\bigg( \frac{m^2}{\pi F} \junder , 0\bigg)
\quad
\text{on}
\quad
D_{\T}(1/m), 
\end{align*}
and that on $\Sph^3 \setminus L$, 
\begin{align*}
\Phi = \Ghat + \Phat + \Phip, 
\quad
E' = - \Lcaltilde (\Ghat + \Phat) = \Lcaltilde \Phip.
\end{align*}
\end{definition}

\begin{lemma}[Estimates on $\Phip$]
\label{LPhip3}
The following hold.
\begin{enumerate}[label=\emph{(\roman*)}]
	\item $\|{\Phi':C^k(\Sph^3  ,\tilde{g})}\|\leq C(k) m$.
	\item $\| \Phip - \Phip(p_0) :C^k_\sym( D_L(\delta), \tilde{g}, m^2 \dbold^2_L) \| \leq C(k) m$.
\end{enumerate}
\end{lemma}
\begin{proof}
Item (i) is proved in \cite[Lemma 4.24]{KapZou}, but we sketch the argument for clarity: first, using Lemma \ref{Lgreen3}(iii) and \eqref{EGhat4}, it follows that
\begin{align}
\label{EGhat4}
 \|\Ghat:C^k(\Sph^3 \setminus D_L(\delta),\gtilde)\|\leq C(k) m.
\end{align}
By definition \ref{ddecomp3}, on $\Sph^3 \setminus D_L(\delta)$, we have
\begin{align*}
E' = - \Lcaltilde \Ghat + \Lcaltilde \Psibold[ 1/2m, 1/m ; \dbold_\T] \Big( \frac{m^2}{ \pi F} \junder, 0\Big), 
\end{align*}
and using \eqref{EGhat4}, Lemma \ref{Lphiave3}(iii), and basic ODE theory for $\junder$, we find
\begin{align}
\label{Epest4}
\|E':C^k(\Sph^3  ,\gtilde)\|\leq C(k) m.
\end{align}
Next, because $\Lcaltilde \Phip_\osc = E'_\osc$,  by arguing as in \cite[Lemma 2.17]{KapZou} and using \eqref{Epest4}, it follows that 
\begin{align*}
\| \Phip_\osc : C^k(\Sph^3, \gtilde) \| \leq C(k) m. 
\end{align*}
To complete the proof of (i), it suffices to prove the desired estimate for $\Phip_\ave$.  For this, note from Definition \ref{ddecomp3} that $\Lcaltilde \Phip_\ave = E'_\ave$, which amounts to an ODE.  The desired estimate then follows from \eqref{Epest4}. 

For item (ii), note by the symmetries that $\Phip - \Phip(p_0)$ has vanishing value and differential at each point of $L$.  The estimate (ii) then follows from this, Taylor's theorem, and the estimate in (i).
\end{proof}

\begin{lemma}
\label{Lphitau3}
With $\Phi$ as in Lemma \ref{Lphi3} and $\varphi \in C^\infty_{\sym} ( \Sph^3\setminus L)$ defined by 
\begin{align}
\label{Etau3}
\varphi : = \tau^2 \Phi,
\quad 
\text{where}
\quad
\tau = \left( \frac{m^2}{\pi F} + \Phip(p_0)\right)^{-1},
\quad
\end{align}
and with $A: = D_L(2\tau) \setminus D_L(\tau/2)$, the following estimate holds. 
\begin{align*}
\| \varphi + \tau^2/ \dbold_L - \tau : C^k( A, \tau^{-2} g) \| \leq C(k)\tau^{5/2}.
\end{align*}
\end{lemma}
\begin{proof}
Expanding $\varphi$ and $\Phi$ using \eqref{Etau3} and Definition \ref{ddecomp3}, noting from Definition \ref{ddecomp3} that $\Phat = \frac{m^2}{\pi F} \phiunder$ on a neighborhood of $\T$ in $\Sph^3$, we see that
\begin{align}
\label{Ephi33}
\varphi + \tau^2/ \dbold_L - \tau = \tau^2 \left( G \circ \dbold_L + 1/ \dbold_L + \frac{m^2}{\pi F} ( \phiunder - 1) +  \Phip - \Phip(p_0)\right)
\end{align}
on $A$.  By Lemma \ref{Lgreen3}(iii) and Lemma \ref{Lphiave3}(iv), we have
\begin{align*}
\begin{gathered}
\| G\circ \dbold_L + 1/ \dbold_L : C^k(A, \tau^{-2} g) \| \leq C(k) \tau, 
\\
\| \phiunder - 1 : C^k(A, \tau^{-2} g ) \| \leq C(k)\tau^2. 
\end{gathered}
\end{align*}
Next, using Lemma \ref{LPhip3}(ii) and recalling the definitions, we have
\begin{align*}
\| \Phip - \Phip(p_0) : C^k( A, \tau^{-2} g) \| 
&\leq \| \Phip - \Phip(p_0) : C^k(A, \gtilde) \| 
\\
&\leq C(k) m^3 \tau^2 
\\
&\leq C(k) \tau^{1/2},
\end{align*}
where we have used that $\tau$ is uniformly comparable to $m^{-2}$.  By combining these estimates with the expansion \eqref{Ephi33}, the conclusion follows. 
\end{proof}

\subsection{The linearized equation}

\begin{definition}
\label{dAD4}
We define weighted norms $\| \cdot \|_{k, \alpha; \Omega}$ just as in Definition \ref{dnorms}, define domains $D$ and $A$ of $\Sph^3$ by
\begin{align*}
D : =  D_{L} (\tau), 
\quad
A: = D_{L}(2 \tau) \setminus D_{L}(\tau/2),
\end{align*}
define a metric $\ghat$ on $A$ by $\ghat = \tau^{-2} g$,  let $\nuhat$ denote the $\ghat$-unit outward pointing normal to $D$ along $\partial D$, so that $\nuhat = \partial_{\rhat}|_{\partial D}$. Finally, for convenience, denote by $r: = \dbold_L$ on $A$.
\end{definition}

\begin{lemma}
\label{Lphierr4}
With $\varphi$ the LD solution from Lemma \ref{Lphitau3}, the function 
\begin{align}
\label{Ephierr4}
\phierr : = \varphi + \tau^2/\dbold_L -\tau \in C^{\infty}_{\sym}(A)
\end{align}
satisfies $\| \phierr \|_{3, \alpha; A} \leq C\tau^{5/2}$.
\end{lemma}
\begin{proof}
This is straightforward from Lemma \ref{Lphitau3}, using Definition \ref{dnorms}.
\end{proof}

We use the subscripts ``avg" and ``osc" just as in Notation \ref{Nhigh}.  Note that because all the components of $\partial D$ are equivalent up to symmetry, we will not need to use the ``low" and ``high" decompositions in Notation \ref{Nhigh}.

In particular, we have
\begin{align*}
L^2_{\sym}(\partial D) &= L^2_{\sym, \ave} (\partial D) \oplus L^2_{\sym, \osc}(\partial D)
\end{align*}

\begin{lemma}
\label{Lext04}
There is a bounded linear map
\begin{align*}
\Hdelta : C^{2, \alpha}_{\sym, \osc}(\partial D) \rightarrow C^{2, \alpha}_{\sym}(D_L(1/m) \setminus D)
\end{align*}
with the properties that  
\begin{enumerate}[label=\emph{(\roman*)}]
\item $\Hdelta$ restricts to the identity on $\partial D$.
\item $\Delta_{\gcir} \Hdelta v = 0$ on $ D_L(1/m) \setminus D$. 
\item $\| \Hdelta v : C^{2, \alpha}_{\sym}(D_L(1/m) \setminus D, \dbold_L, g, \rhat^{-2})\| \leq C \| v \|_{2, \alpha; \partial D}$.
\item $\partialnu \Hdelta - 2$ has a bounded inverse
	\begin{align*}
		(\partialnu \Hdelta- 2)^{-1} : C^{1, \alpha}_{\sym, \osc}(\partial D) \rightarrow C^{2, \alpha}_{\sym, \osc}(\partial D).
	\end{align*}
\end{enumerate}
\end{lemma}
\begin{proof}
We omit some of the details of the proof, which are very similar to those of Lemma \ref{Lext0}.  For (iii) and (iv), first recall that if $p$ is a homogeneous harmonic polynomial on $\R^3$ of degree $k$, then the Kelvin transform $K[p]$ of $p$ satisfies $K[p] = |x|^{-1-2k} p$.  From this and arguing as in the proof of Lemma \ref{Lext0}, (iii) follows, again using that $v_\ave \in \Hcal^0_\sym(\partial D)$ is zero and $\Hcal^1_\sym(\partial D)$ is trivial by the $\group$-symmetry. 

Because of the properties of the Kelvin transform above, it follows that
\begin{align*}
\partialnu H_\Delta v_k = (k+1) \frac{\tau}{\sin \tau} v_k
\end{align*}
whenever $v_k \in \Hcal^k_{\sym}(\partial D)$.  Consequently,  is easy to see from the symmetry that the smallest eigenvalue of $\partialnu H_\Delta -2$ is bounded away from zero, and it follows by standard theory that $\partialnu H_\Delta -2$ has a bounded inverse.  
\end{proof}

\begin{prop}
\label{Pext14}
There is a bounded linear map
\begin{align*}
H_{\Lcal} : C^{2, \alpha}_{\sym, \osc}(\partial D) \rightarrow C^{2, \alpha}_{\sym}(\Sph^3 \setminus D)
\end{align*}
such that if $v \in C^{2, \alpha}_{\sym, \osc}(\partial D)$ and $u:= H_{\Lcal} v$, then the following hold.
\begin{enumerate}[label=\emph{(\roman*)}]
\item $\Lcal u  = 0$ on $\Sph^2 \setminus D$.
\item $(u|_{\partial D})_{\osc} = v$ and $|(u|_{\partial D})_{\ave}| \leq \frac{C}{\sqrt{m}} \| v \|_{2, \alpha; \partial D}$.
\item $\| \partialnu H_\Lcal v -\partialnu \Hdelta v \|_{1, \alpha; \partial D} \leq \frac{C}{\sqrt{m} } \| v \|_{2, \alpha; \partial D}$.
\end{enumerate}
\end{prop}
\begin{proof}
The proof is split into three steps.  \newline
\emph{Step 1: the approximate extension operator $\Happr$}.  Define a linear map
\begin{align*}
\Happr : C^{2, \alpha}_{\sym, \osc}(\partial D) \rightarrow C^{2, \alpha}_{\sym}(\Sph^3 \setminus D)
\end{align*}
by $\Happr:=\Htide + \Htider$, where $\Htide $ and $\Htider$ are defined by requesting that
\begin{equation}
\label{Ehpp4}
\begin{gathered}
\mathrm{supp}\,  \Htide v \subset D_L(1/m) \setminus D \quad \text{and} 
\\
\Htide  v = \Psibold\left[ \textstyle{\frac{1}{2m}}, \frac{1}{m}; \dbold_{L} \right] \left( \Hdelta v,  0\right)
\quad
\text{on}
\quad
D_{L}(1/m) \setminus D,
\end{gathered}
\end{equation}
\begin{equation}
\begin{gathered}
\label{Ehtp4}
\Lcal \Htider v = - E,
\quad
\text{on}
\quad 
\Sph^3, 
\quad
\text{where}
\\
E = \Lcal \Htide v \text{ on } \Sph^3 \setminus D,  \quad \Delta E = 0 \text{ on } D, \quad E|_{\partial D} = (\Lcal \Htide v)|_{\partial D}
\end{gathered}
\end{equation}
whenever $v \in C^{2, \alpha}_{\sym, \osc}(\partial D)$.
In particular, notice that $E$ is a well-defined element of $C^{0, \alpha}_{\sym}(\Sph^3)$ due to the smallness of $D$ and the choice of boundary values, and that the existence, uniqueness, and linearity of $\Htider$ is ensured by the Fredholm alternative and standard linear theory, since $E$ is $\group$-symmetric element of $C^{0, \alpha}(\Sph^3)$. 

\noindent \emph{Step 2: estimates for $\Happr$}. First note that
\begin{align}
\label{Eutp14}
\begin{cases}
\Lcal \Happr v  = 0 \hfill \quad &\text{in} \quad \Sph^3 \setminus D\\
\phantom{\Lcal} \Happr v= v+ \Htider v \hfill \quad &\text{on} \quad \partial D.
\end{cases}
\end{align}
We next estimate $\Htider v$. For ease of notation, in what follows we denote 
\[ 
\Omega : = D_L(1/m) \setminus D_L(1/2m).
\] 
We first estimate $E$ on $\Omega$, where it satisfies $E = \Lcal H_\Delta v= (\Delta +3) H_\Delta v$.  Using \eqref{Ehpp4} and that $\Delta_{\gcir} \Hdelta v = 0$ from Lemma \ref{Lext04}, we estimate
\begin{align*}
\| E\|_{0, \alpha; \Omega} \leq C \| (\Delta_g - \Delta_{\gcir}) \Hdelta v \|_{0, \alpha; \Omega} + C\| \Hdelta v\|_{0, \alpha; \Omega}.
\end{align*}
Estimating the difference of the Laplacians directly shows that
\begin{align*}
\| (\Delta_g - \Delta_{\gcir} ) H_\Delta v\|_{0, \alpha;  \Omega} \leq C \| H_\Delta v \|_{2, \alpha; \Omega} \leq C \| v \|_{2, \alpha; \partial D},
\end{align*}
where we have used Lemma \ref{Lext0}(iii).  Consequently
\begin{align}
\label{Eharmd}
\| E \|_{0, \alpha; \Omega} \leq C \| v \|_{2, \alpha; \partial D}, 
\end{align}
and using this with the the definition of $E$ on $D$ from \eqref{Ehtp} implies
\begin{align*}
\| E\|_{L^\infty(D) } = \| E \|_{L^\infty( \partial D) } \leq C \| v \|_{2, \alpha; \partial D}. 
\end{align*}

In similar fashion, using \eqref{Ehpp4} and the definitions, we have
\begin{align*}
\| E\|_{0, \alpha; D_L(1/m) \setminus D_L(1/2m)} \leq Cm^2 \| H_\Delta v \|_{2, \alpha;  D_L(1/m) \setminus D_L(1/2m)} \leq C \| v \|_{2, \alpha; \partial D},
\end{align*}
where the second inequality uses the decay estimate in Lemma \ref{Lext0}(iii).

In total, we have $\| E \|_{L^\infty(\Sph^3)} \leq C \| v \|_{2, \alpha; \partial D}$, and furthermore
\begin{align*}
\| E \|_{L^2(\Sph^3)} \leq \| E\|_{L^\infty(\Sph^3)} | D_L(1/m)  |^{1/2} \leq \frac{C}{\sqrt{m}} \| v \|_{2, \alpha; \partial D},
\end{align*}
where the last inequality estimates the area of $D_L(1/m)$, which consists of $m^2$ balls, each with radius $1/m$ and measure bounded by $C/m^3$. 

Recalling that $\Lcal \Htider v = - E$ from \eqref{Ehtp4}, combining the preceding with standard elliptic theory and De Giorgi-Nash-Moser theory implies that
\begin{align*}
\| \Htider v\|_{C^0(\Sph^3)} 
\leq C \| \Htider v \|_{L^2(\Sph^3)}
\leq C \| E \|_{L^2(\Sph^3)} 
\leq \frac{C}{\sqrt{m}} \| v \|_{2, \alpha; \partial D}.
\end{align*}

We now obtain $C^{2, \alpha}$ estimates on $\Htider v$, first on $A$.  Using that $A$ has uniformly bounded geometry in the metric $\tau^{-2} g$ and that $\tau^2 \Lcal \Htider v = \tau^2 E$, Schauder theory and the above estimates imply
\begin{align*}
\| \Htider  v \|_{2, \alpha; A} &\leq C ( \| \Htider  v \|_{C^0(A)} + \| \tau^2 E \|_{0, \alpha; A}) 
\leq \frac{C}{\sqrt{m}} \| v \|_{2, \alpha; \partial D},
\end{align*}
where the second inequality uses the estimates established above alongside the relationship between $\tau$ and $m$.  Arguing analogously for other subdomains and combining shows that
\begin{align}
\label{Etpv4}
\| \Htider  v \|_{2, \alpha; \Sph^3 \setminus D} \leq \frac{C}{\sqrt{m}} \| v \|_{2, \alpha; \partial D}.
\end{align}

We now collect the estimates for $\Happr$ we need: \eqref{Eutp14} and \eqref{Etpv4} imply
\begin{align}
\label{Ehapperup4}
\| \Happr v - v \|_{2, \alpha; \partial D} \leq \frac{C}{\sqrt{m}} \| v \|_{2, \alpha; \partial D}, 
\end{align}
and combining Lemma \eqref{Ehpp4}, \eqref{Etpv4}, and \ref{Lext04}, shows that
\begin{align}
\label{Ehhhigh3}
\| \Happr v \|_{2, \alpha; \Sph^2 \setminus D} \leq C \| v \|_{2, \alpha; \partial D}, 
\end{align}
and the definitions imply $\partialnu \Happr v - \partialnu \Hdelta v = \partialnu \Htider  v$, 
so \eqref{Etpv4} shows
\begin{align}
\label{Eh004}
\| \partialnu \Happr v - \partialnu \Hdelta v \|_{1, \alpha; \partial D} \leq \frac{C}{\sqrt{m}} \| v \|_{2, \alpha; \partial D}. 
\end{align}
\noindent \emph{Step 3: the exact extension operator}.  
Just as in Step 5 of Proposition \ref{Pext1}, we observe that \eqref{Ehapperup4} implies $P_{\osc}  \iota^*  \Happr $ has a uniformly bounded inverse.  We then define $H_\Lcal = \Happr (P_{\osc}  \iota^*  \Happr )^{-1}$.  The proposition now follows from the preceding estimates, just as in the proof of Proposition \ref{Pext1}.
\end{proof}

\begin{prop}
	\label{Pinhomog4}
	There is a linear map
	\begin{align*}
		J_\Lcal :  \left\{ E \in C^{0, \alpha}_{\sym}(\Sph^3 \setminus D) : \mathrm{supp} \, E \subset A \setminus D\right\} 
		\rightarrow C^{2, \alpha}_{\sym}(\Sph^3 \setminus D)
	\end{align*}
	such that if $E$ is in the domain of $J_\Lcal$ and $u= J_{\Lcal} E$, then the following hold. 
	\begin{enumerate}[label=\emph{(\roman*)}]
		\item $\Lcalhat u = E$ on $\Sph^3 \setminus D$. 
		\item $\| u \|_{2, \alpha; \Sph^3 \setminus D} \leq C \| E\|_{0, \alpha; A \setminus D}$. 
		\item $(u|_{\partial D})_\osc = 0$ and  $|(u|_{\partial D})_{\ave}|   \leq C \| E \|_{0, \alpha; A \setminus D}$.
	\end{enumerate}
\end{prop}
\begin{proof}
	We first decompose $E = E_\low + E_\high$, where $E_\low$ is constant on each circle $\partial D_{L} (r)$ where $r \in (\tau, 2 \tau)$ and define $J_\Lcal E_\low$ and $J_\Lcal E_\high$ separately.
	
	There is a unique ODE solution  $u_\low \in C^{2, \alpha}_{\sym}(A \setminus D)$ depending only on $\dbold_{L}$ on each $A \setminus D$ solving $\hat{\Lcal} u_\low =E_\low$ with the initial conditions
	\begin{align*}
		u|_{\partial D_{L}} ( 2\tau) = 0, 
		\quad
		\partial_{r} u |_{\partial D_{L}}(2 \tau) = 0. 
	\end{align*}
	In particular, these conditions and the assumption $\mathrm{supp} E \subset A \setminus D$ implies $u_\low$ can be considered smooth on all $\Sph^3 \setminus D$ and supported on $A \setminus D$, and basic ODE theory implies
	\begin{align}
		\label{Euode4}
		\| u_\low \|_{2, \alpha; \Sph^3\setminus D} \leq C \| E_\low \|_{0, \alpha; A \setminus D}.
	\end{align}

Next, we will obtain a function $u_\high \in C^{2,\alpha}_{\sym}(D(1/m) \setminus D)$ which solves $\Delta_{\tau^{-2}\gcir} u_{\high} = E_{\high}$ and satisfies the estimate
	\begin{align}
		\label{Eupest14}
		\| u_{\text{high}}: C^{2, \alpha}_{\sym}(D_L(1/m) \setminus D, r, g, \rhat^{-3})\|
		\leq 
		C 
		\| E_{\text{high}}\|_{0, \alpha; A \setminus D}.
	\end{align}
To see this, we argue as follows.  First, in the Euclidean metric $\gcir$, by the symmetries $D(1/m)$ can be identified with a single Euclidean ball.  Through the Kelvin transform $K$, which is defined by $K[ u] = |x|^{2-n} u(x/|x|^2)$, and satisfies $\Delta ( K[u]) = K(|x|^4 \Delta u)$ for $\Delta$ the standard Laplacian on $\R^n$ (here $n=3$), the equation $\Delta_{\tau^{-2}\gcir} u_{\high} = E_{\high}$ on the exterior of $D$ corresponds to an equivalent equation $\Delta_{\tau^{-2}\gcir} \utilde_{\high} = \Etilde_{\high}$ on the interior of $D$. By standard theory, the restricted support of $E_{\high}$, and the symmetries, there is a unique $C^{2,\alpha}$ solution $\utilde_{\high}$ defined on $D$ solving the preceding equation with estimate
\begin{align*}
\| \utilde_{\high} : C^{0}_{\sym}( D, \gcir, \rhat^2) \| \leq C\| E_{\high}\|_{0, \alpha; A \setminus D}. 
\end{align*}
Using the Kelvin transform, the corresponding solution $u_\high$ to the equation $\Delta_{\tau^{-2} \gcir} u_{\high} = E_{\high}$ on $D_L(1/m) \setminus D$ then satisfies
\begin{align*}
\| u_{\high} : C^0( D_L(1/m) \setminus D, \gcir, \rhat^{-3}) \| \leq C\| E_{\high}\|_{2, \alpha; A\setminus D},
\end{align*}
and \eqref{Eupest14} follows from this by standard regularity theory and the

	We then define $\Jtilde_\Lcal E = u_\low +  \utilde_{\high} + \utilde^{err}_{\high}$, where $\utilde_{\high}, \utilde_{\high}^{err} \in C^{2, \alpha}_{\sym}(\Sph^3 \setminus D)$ are defined by requesting that
	\begin{equation*}
		\begin{gathered}
			\mathrm{supp} \, \utilde_{\high} \subset D_L(1/m) \setminus D 
			\quad
			\text{and}
			\\
			\utilde_{\high} =  \Psibold\left[ \textstyle{ \frac{1}{2m}}, \frac{1}{m}; \dbold_{L} \right]  \left( u_\high,  0\right)
			\quad
			\text{on} 
			\quad
			D_L(1/m) \setminus D, 
		\end{gathered}
	\end{equation*}
	\begin{equation*}
		\begin{gathered}
			\Lcal \utilde^{err}_{\high} = - \Etilde 
			\quad
			\text{on}
			\quad \Sph^3
			\quad \text{where}
			\\
			\mathrm{supp} \, \Etilde \subset D_L(1/m) \setminus D_L(1/2m),
			\quad \Etilde = \Lcal u_{\high}.
		\end{gathered}
	\end{equation*}
	From this definition, it follows that $\Lcal \Jtilde_\Lcal E = E$ on $\Sph^3 \setminus D$.  Moreover,  \eqref{Euode4}, \eqref{Eupest14}, and the bound
	\begin{align*}
		\| \utilde^{err}_{\high} \|_{2, \alpha; \Sph^3 \setminus D} \leq \frac{C}{\sqrt{m}}\| E\|_{0, \alpha; A \setminus D}
	\end{align*}
	which follows by arguing as in the proof of \ref{Pext14} show that $\Jtilde_\Lcal$ is bounded. 
	
	Finally, we define 
	\begin{align}
		\label{EJdef4}
		J_{\Lcal} E := \Jtilde_\Lcal E  - H_{\Lcal} (( \Jtilde_\Lcal E)|_{\partial D})_\osc
	\end{align}
	where $H_\Lcal$ is as in Proposition \ref{Pext14}.  Item (i) follows from the definition of $\tilde{J}$ and Proposition \ref{Pext14}(i).  Item (ii) follows by Proposition \ref{Pext14}(ii). 
	Finally, (iii) follows from Proposition \ref{Pext14}(iii). 
\end{proof}

\begin{prop}
\label{Plin4}
The map $\Bcal : C^{2, \alpha}_{\sym, \osc}(\partial D) \rightarrow C^{1, \alpha}_{\sym, \osc}(\partial D)$ defined by
\begin{equation}
\label{EBcal4}
\Bcal v := (\hat{\nu} H_\Lcal v - 2v)_{\osc}
\end{equation}
has a bounded right inverse $\Rcal : C^{1, \alpha}_{\sym, \osc}(\partial D)  \rightarrow C^{2, \alpha}_{\sym, \osc}(\partial D)$; 
that is, $\Bcal \Rcal$ is the identity map on $C^{1, \alpha}_{\sym, \osc}(\partial D)$.
\end{prop}
\begin{proof}
By Lemma \ref{Lext04}(iv), the operator $\Rcalappr : C^{1, \alpha}_{\sym, \osc}(\partial D) \rightarrow C^{2, \alpha}_{\sym, \osc}(\partial D)$ defined by $\Rcalappr =  (\partialnu \Hdelta -2)^{-1}$ is bounded. 

Now let $E \in C^{1, \alpha}_{\sym, \osc}(\partial D)$ be given, and set $v := \Rcalappr E$. Based on the definitions of $E$ and $v$, we have $\partialnu \Hdelta v - 2v = E$.  Consequently, 
\begin{align*}
\partialnu H_\Lcal v - 2v=\partialnu H_\Lcal v  -\partialnu \Hdelta v + E,
\end{align*}
and taking oscillatory parts and rearranging shows that
\begin{align}
\label{EE1ERR}
\Bcal v - E =  (\partialnu H_\Lcal v - \partialnu \Hdelta v )_\osc.
\end{align}
From \eqref{EE1ERR}, Proposition \ref{Pext14}(iii), and the boundedness of $\Rcalappr$, we conclude
\begin{equation}
\label{EBv1}
\begin{aligned}
\| \Bcal v- E\|_{1, \alpha; \partial D} \leq  \frac{C}{\sqrt{m}} \| v \|_{2, \alpha; \partial D} \leq  \frac{C}{\sqrt{m}} \| E\|_{1, \alpha; \partial D}.
\end{aligned}
\end{equation}

Thus, the operator $\Bcal \Rcalappr - 1$ on $C^{1, \alpha}_{\sym, \osc}(\partial D)$ has norm bounded by $C/ \sqrt{m}$, hence $\Bcal \Rcalappr$ is a perturbation of the identity, with bounded inverse.  The proof is completed by defining $\Rcal = \Rcalappr (\Bcal \Rcalappr)^{-1}$ and using the preceding facts. 
\end{proof}

\subsection{Perturbations of $\Sph^3 \setminus D$} Given $v \in C^{2, \alpha}_{\sym}(\partial D)$ with $\| v \|_{2, \alpha; \partial D}< \tau^2$, we define the perturbation $D_v = \Sph^2 \setminus D$ just as in the beginning of Section \ref{Spert}, and we keep to the same notation introduced in \ref{Notv}.

\begin{lemma}
\label{Lphibdest4}
$\varphi_v$ is a $C^1$ function of $v$ and satisfies the following.
\begin{enumerate}[label=\emph{(\roman*)}]
\item $\| \varphibd_0\|_{2, \alpha; \partial D} \leq C\tau^{5/2}$. 
\item $\| (\varphibd_v)'(v)h + h \|_{2, \alpha; \partial D} \leq C\tau \| h \|_{2,\alpha; \partial D}$.
\item $ \|  (\varphi^\partial_{v, \osc})'(v)h \|_{2, \alpha; \partial D} \leq \| h_\osc \|_{2, \alpha; \partial D} + C \tau \| h\|_{2, \alpha; \partial D}$.
\item $\| \varphibd_{v, \osc}\|_{2, \alpha; \partial D} \leq \| v_\osc \|_{2, \alpha; \partial D} + C \tau \| v\|_{2, \alpha; \partial D} + C\tau^{5/2}$. 
\end{enumerate}
\end{lemma}
\begin{proof}
Recalling Lemma  \ref{Lphierr4} and Notation \ref{Notv}, we have
\begin{align}
\label{EphibdT4}
\varphibd_v = \tau - \tau(1-\vhat)^{-1} + (\varphi^{err}_v)^\partial 
\end{align}
From this, we have $\varphi^{\partial}_0 = \phierr|_{\partial D}$, and (i) then follows from Lemma \ref{Lphierr4}.

Next, recalling that $\phierr_v = (\exp V_v)^* \phierr$, we see from \eqref{EphibdT} and a direct calculation that 
	\begin{align*}
	(\varphibd_v)'(v) h =  - h(1-\vhat)^{-2} -  \hat{h} \iota^* (\exp V_v)^* \partial_{\rhat} \phierr,
	\end{align*}
and after estimating that
\begin{align*}
\| (\varphibd_v)'(v) h + h \|_{2, \alpha; \partial D} \leq C \|h\|_{2, \alpha; \partial D} ( \| \vhat \|_{2, \alpha; \partial D} + \| \hat{\varphi}^{err}\|_{3, \alpha; A}).
\end{align*}	
Using the bound $\| v \|_{2, \alpha; \partial D} < \tau^2$ and the estimates on $\phierr$ in Lemma \ref{Lphierr4}, we conclude (ii).  Item (iii) follows from (ii), and (iv) follows from (i), (iii), and the mean value inequality. 
\end{proof}

Next, notice that Lemma \ref{Llapdiff} and its proof hold verbatim in this setting. 

\begin{corollary}
\label{CJ4}
Assuming $\| v \|_{2, \alpha; \partial D} < \tau^2$, there is a bounded linear map
\begin{align}
\label{EJLcalv4}
\begin{gathered}
J_{\Lcal_v} :  \left\{ E \in C^{0, \alpha}_{\sym}(\Sph^3 \setminus D) : \mathrm{supp} \, E \subset A \setminus D)\right\} 
\rightarrow C^{2, \alpha}_{\sym}(\Sph^3 \setminus D), 
\\
J_{\Lcal_v} := J_{\Lcal} [ 1-  (\Lcalhat - \Lcalhat_v)J_\Lcal]^{-1}
\end{gathered}
\end{align}
such that if $E \in \mathrm{dom}\,  J_{\Lcal_v}$ and $u= J_{\Lcal_v} E$, then the following hold. 
\begin{enumerate}[label=\emph{(\roman*)}]
\item $\Lcalhat_v u = E$ on $\Sph^2 \setminus D$. 
\item $(u|_{\partial D})_\osc = 0$ and  $|(u|_{\partial D})_{\ave}|   \leq C \| E \|_{0, \alpha; A \setminus D}$.
\end{enumerate}
\end{corollary}
\begin{proof}
Exactly the same as the proof of Corollary \ref{CJ}, but using \ref{Pinhomog4}.
\end{proof}

\begin{lemma}
\label{Lphiv4}
Whenever $\| v \|_{2, \alpha; \partial D} < \tau^2$, the function 
\begin{equation}
\label{Ephiv04}
\begin{aligned}
\phi_v : = \varphi_v - \xiv - \xivtilde \in C^{2, \alpha}_{\sym}(\Sph^3 \setminus D),
\end{aligned}
\end{equation}
where $\xiv, \xivtilde \in C^{2, \alpha}_{\sym}(\Sph^3 \setminus D)$ are defined by
\begin{align*}
\xiv: = H_\Lcal \varphibd_{v, \osc},
\quad
\xivtilde : = J_{\Lcal_v} ( \Lcalhat - \Lcalhat_v) \xiv,
\end{align*}
is constant on $\partial D$ and satisfies $\Lcal_v \phi_v =0$ on $\Sph^3 \setminus D$. 
\end{lemma}
\begin{proof} 
Exactly the same as the proof of Lemma \ref{Lphiv}.
\end{proof}

\begin{lemma}[Estimates at the boundary]
\label{Lphibd4}
If $\| v \|_{2, \alpha; \partial D} < \tau^2$, then 
\begin{enumerate}[label=\emph{(\roman*)}]
\item $|(\xi^\partial_{v,\ave})'(v)h | \leq  \frac{C}{\sqrt{m}}( \| h_\osc \|_{2, \alpha; \partial D} + \tau \| h \|_{2, \alpha; \partial D})$.
\item $| (\xitilde^{\partial}_{v,\ave})'(v) h |  \leq C \tau \| h \|_{2, \alpha; \partial D}$. 
\item $| (\phi_v^\partial)'(v)h + h_\ave| \leq C ( \tau\| h \|_{2, \alpha; \partial D} + \frac{1}{\sqrt{m}} \| h_\osc\|_{2, \alpha; \partial D})$. 
\end{enumerate}
\end{lemma}
\begin{proof}
Completely analogous to the proof of Lemma \ref{Lphibd}, but using Proposition \ref{Pext14}, Lemma \ref{Lphibdest4}, and Corollary \ref{CJ4} instead of \ref{Pext1}, \ref{Lphibdest}, and \ref{CJ}.
\end{proof}

\begin{prop}
\label{Pphiv4}
There is a $C^1$ map
\begin{align*}
f: \{ w \in C^{2, \alpha}_{\sym, \osc}(\partial D) : \| w \| < \tau^2/2 \}   \rightarrow \R
\end{align*}
 uniquely determined by the property that $\phi_{w + f(w)}$ vanishes on $\partial D$.  Moreover,
 \begin{enumerate}[label=\emph{(\roman*)}]
\item $|f(0) | \leq C \tau^{5/2}$, and
\item $|f(w)-f(0)| \leq  \frac{C}{\sqrt{m}}\|w\|_{2, \alpha;\partial D}$.
\end{enumerate}
\end{prop}
\begin{proof}
Identical to the proof of Lemma \ref{Pphiv}, but using Lemmas \ref{Lphiv4} and \ref{Lphibd4} instead of \ref{Lphiv} and \ref{Lphibd}.
\end{proof}

\begin{prop}
\label{LinN4}
The map 
\begin{align*}
&N : \{ w \in C^{2, \alpha}_{\sym, \osc}(\partial D) : \| w \|_{2, \alpha; \partial D} < \tau^2/2 \}  \rightarrow C^{1, \alpha}_{\sym, \osc}(\partial D),
\\
&N( w) = ( \nuhat_v \phi_v)_{\osc}, \quad \text{where} \quad v: = w+f(w)
\end{align*}
and $f$ is as in Proposition \ref{Pphiv4}, satisfies the following. 
\begin{enumerate}[label=\emph{(\roman*)}]
\item $\| N(0) \|_{1, \alpha; \partial D } \leq C \tau^{5/2}$. 
\item $\| N(w) - N(0) -  \Bcal w\|_{1, \alpha; \partial D} \leq C \tau^3$, where $\Bcal$ is as in \eqref{EBcal4}.
\end{enumerate}
\end{prop}
\begin{proof}
For  $w, v, \phi_v$ as above, expanding $\phi_{v}$ via \eqref{Ephierr4} and \eqref{Ephiv04} shows that
\begin{equation}
\label{Ephihat4}
\begin{aligned}
\phi_v &=\tau - \tau(1-\zhat - \hat{v})^{-1} + \phierr_{v} - H_\Lcal \varphibd_{v, \osc} - \xivtilde,
\\
\partialnu \phi_v &= -\tau(1-\vhat)^{-2}+ \partialnu \phierr_{v} - \partialnu H_\Lcal \varphibd_{v, \osc} - \partialnu  \xivtilde,
\end{aligned}
\end{equation}
where $z$ is the signed distance from $\partial D$ as in Lemma \ref{Llapdiff}. 

Arguing just as in the proof of Proposition \ref{LinN}, we find that
\begin{equation}
\label{EN04}
\begin{gathered}
N(0) =  (\partialnu \phierr_{\cunder} - \partialnu H_\Lcal \varphibd_{\cunder, \osc} - \partialnu \xitilde_{\cunder})_{\osc}, 
\\
\| N(0) \|_{1, \alpha; \partial D} 
\leq  C\| \phierr \|_{3, \alpha; A} + C\| \varphibd_{\cunder, \osc}\|_{2, \alpha; \partial D},
\end{gathered}
\end{equation}
and (i) follows from this using Lemma  \ref{Lphierr4} and Lemma \ref{Lphibdest4}(iv).

For (ii), since $\phi_v$ vanishes along $\partial D$, it follows that
 it follows that
\begin{align*}
(\hat{\nu}_v - \hat{\nu}) \phi_v  = ( \langle \hat{\nu}_v, \hat{\nu}\rangle_{\ghat} - 1 )  \partialnu \phi_v ,
\end{align*}
and from this and Lemma \ref{Llapdiff}(i) that 
\begin{equation}
\begin{aligned}
\label{ENv14}
\| N(w) - (\partialnu \phi_v)_{\osc} \|_{1, \alpha; \partial D} 
\leq C \tau^2  \| \partialnu \phi_v \|_{1, \alpha; \partial D}
\leq C \tau^3,
\end{aligned}
\end{equation}
where the second inequality estimates $\partialnu \phi_v$ using \eqref{Ephihat4}.

Using \eqref{EBcal4}, \eqref{Ephihat4}, and \eqref{EN04}, we  find 
\begin{multline*}
(\partialnu \phi_v)_{\osc} - \Bcal w - N(0) = (-\tau(1-\vhat)^{-2} + \partialnu \phierr_{v} - \partialnu H_\Lcal \varphibd_{v, \osc} - \partialnu  \xivtilde
\\
-\hat{\nu} H_\Lcal w + 2w -\partialnu \phierr_{\cunder} + \partialnu H_\Lcal \varphibd_{\cunder, \osc} + \partialnu \xitilde_{\cunder} )_\osc
\end{multline*}
and using that $w = v - c = v_\osc$ and rearranging, we see
\begin{equation*}
\begin{gathered}
(\partialnu \phi_v)_{\osc} - \Bcal w - N(0) = (I + II - III - IV)_\osc,
\quad
\text{where}
\\
I: =  2v - \tau(1-\vhat)^{-2},
\quad
II: =  \partialnu (\phierr_{v} -  \phierr_{\cunder}),
\\
III: =  \partialnu H_{\Lcal}(\varphibd_v +v - \varphibd_{\cunder} - c)_\osc,
\quad
IV: = \partialnu ( \xivtilde - \xitilde_{\cunder}) .
\end{gathered}
\end{equation*}
First, we estimate
\begin{align*}
\| I_\osc \|_{1, \alpha; \partial D} & \leq C \| \tau \vhat^2 \|_{2, \alpha; \partial D} 
\leq C \tau^3,
\end{align*}
where we have used that  $\|v\|_{2, \alpha; \partial D} \leq C\| w\|_{2, \alpha; \partial D}$ from Proposition \ref{Pphiv4}.

Next by arguing just as in the proof of Theorem \ref{Tmain}, we have
\begin{align*}
\| II \|_{1, \alpha; \partial D}+ \| III\|_{1, \alpha; \partial D} + \| IV \|_{1, \alpha; \partial D} \leq C\tau^3,
\end{align*}
and combining the preceding, we find
\begin{align}
\label{ENv2}
\| (\partialnu \phi_v)_{\osc} - N(0) - \Bcal w \|_{1, \alpha; \partial D} \leq C \tau^3. 
\end{align}
Item (ii) now follows by combining \eqref{ENv1} and \eqref{ENv2}.
\end{proof}

\begin{theorem}
\label{Tmain42}
There is a number $m_0$ such that if $m>m_0$, $L = L[m] \subset \Sph^3$ is the $\group$-invariant set of $m^2$ points as in \ref{dL3}, $\tau$ is as in \eqref{Etau3}, and the neighborhood $D = D_L(\tau)$ of $L$ is as in \ref{dAD4}, then there is a $\group$-invariant function $v \in C^{2, \alpha}(\partial D)$ satisfying
\begin{align*}
\| v \|_{2, \alpha; \partial D} \leq C \tau^{5/2}
\end{align*}
such that the perturbation $\Omega$ of $\Sph^3 \setminus D$ with boundary the normal graph
\begin{align*}
\partial \Omega : = \{ \exp_p( v(p) \nu(p) ) : p \in \partial D\}
\end{align*}
is a $\lambda_1$-extremal domain in $\Sph^3$ with $\lambda_1(\Omega) = 3$.  In particular, $\Omega$ admits a solution to the overdetermined problem \eqref{EoverOG}, is $\group$-symmetric, and has real-analytic boundary, consisting of $m^2$ components. 
\end{theorem}
\begin{proof}
The proof is essentially identical to that of Theorem \ref{Tmain}, but using the estimates proved in this section, so we don't repeat the details.
\end{proof}


\end{document}